\documentclass{amsart}
\usepackage[]{algorithm2e}
\usepackage{amsmath}
\usepackage{amssymb}
\usepackage{amsthm}
\usepackage[english]{babel}
\usepackage{bm}
\usepackage{color}
\usepackage{graphicx}
\usepackage{hyperref}
\usepackage{listings}
\usepackage{mathrsfs}
\usepackage{url}

\newtheorem{theorem}{Theorem}[section]

\newtheorem{problem}{Problem}
\theoremstyle{definition}

\theoremstyle{remark}
\newtheorem*{remark}{Remark}
\newtheorem{proposition}[theorem]{Proposition}

\DeclareMathOperator*{\divergence}{div}

\DeclareMathOperator*{\Grad}{\boldsymbol\nabla}
\DeclareMathOperator*{\Grads}{\boldsymbol\nabla_{\rm sym}}

\DeclareMathOperator{\grads}{\boldsymbol\nabla_s}
\newcommand\RE{\mathbb{R}}
\newcommand{\derivative}[2]{\frac{\partial #1}{\partial #2}}
\newcommand\Huo{H^1_0(\Omega)^d}
\newcommand\Ldo{L^2_0(\Omega)}
\newcommand\Hub{H^1(\B)^d}

\newcommand\Oft{\Omega^f_t}
\newcommand\Ost{\Omega^s_t}
\newcommand\B{\mathcal B}

\renewcommand\u{\mathbf{u}}
\renewcommand\v{\mathbf{v}}
\renewcommand\d{\mathbf{d}}

\renewcommand\c{\mathbf{c}}

\newcommand\n{\mathbf{n}}
\newcommand\w{\mathbf{w}}
\newcommand\x{\mathbf{x}}
\newcommand\z{\mathbf{z}}
\newcommand\s{\mathbf{s}}
\renewcommand\S{\mathbf{S}}
\newcommand\V{\mathbf{V}}
\newcommand\X{\mathbf{X}}
\newcommand\Y{\mathbf{z}}

\newcommand\LL{\boldsymbol{\Lambda}}
\newcommand\ssigma{\boldsymbol\sigma}

\newcommand\F{\mathbb{F}}

\renewcommand\P{\mathbb{P}}

\newcommand\ds{\mathrm{d}\s}
\newcommand\dx{\mathrm{d}\x}

\newcommand\dt{\Delta t}
\newcommand\dr{\delta\rho}
\newcommand\T{\mathcal{T}}
\newcommand\CNM{\textsf{CNm}\xspace}
\newcommand\CNT{\textsf{CNt}\xspace}
\newcommand\BDF{\textsf{BDF2}\xspace}
\newcommand\BD{\textsf{BDF1}\xspace}
\newcommand\dX{\dot\X}

\title[Higher-order time-stepping for FSI problems]
{Higher-order time-stepping schemes for fluid-structure interaction problems }

\author[D. Boffi, L. Gastaldi, and S. Wolf]{}



\email{daniele.boffi@unipv.it}
\email{lucia.gastaldi@unibs.it}
\email{s.wolf@tum.de}

\thanks{$^*$ Corresponding author: Daniele Boffi}

\begin{document}
\maketitle
\centerline{\scshape Daniele Boffi$^*$}
\medskip
{\footnotesize
 \centerline{Dipartimento di Matematica ``F. Casorati'', University of Pavia}
   \centerline{Pavia, Italy}
} 

\medskip

\centerline{\scshape Lucia Gastaldi}
\medskip
{\footnotesize
 \centerline{DICATAM, University of Brescia}
   \centerline{Brescia, Italy}
}

\medskip

\centerline{\scshape Sebastian Wolf}
\medskip
{\footnotesize
 \centerline{Technische Universit\"at M\"unchen (TUM)}
   \centerline{M\"unchen, Germany }
}

\begin{abstract}
We consider a recently introduced formulation for fluid-structure interaction
problems which makes use of a distributed Lagrange multiplier in the spirit of
the fictitious domain method. In this paper we focus on time integration methods
of second order based on backward differentiation formulae and on the Crank--Nicolson 
method. We show the stability properties of the resulting method; numerical tests confirm the theoretical results.
\end{abstract}
\maketitle
\section{Introduction}

We discuss a scheme involving a fictitious domain approach with a distributed Lagrange multiplier for the modeling of fluid-structure interaction problems~\cite{BCG15,BG17}, which has originated as a natural evolution of the Finite Element Immersed Boundary Method introduced and studied in~\cite{BGH07,BGHP08,BCG11,BG16}. This investigation started from the framework of Peskin's research~\cite{P02} who introduced the (finite difference) Immersed Boundary Method for the modeling of fluid-structure interaction problems.

The project described in this paper has been carried on during the Master thesis of the third author who spent a semester in Pavia within an exchange program between TUM and Pavia.
The aim of the project was to investigate and analyze higher order time schemes for our fluid-structure interaction numerical approach which, so far, had been presented only in combination with low order Euler schemes. On the other hand, in the case of \emph{thick} (codimension zero) solids, the regularity of the solution allows for a convergence in space higher than first order, so that it may pay off to make use of higher order schemes.

The main result of this paper consists in the implementation and in the stability analysis for the second order Backward Differentiation Formula \BDF and for the Crank--Nicolson scheme. We present two possible variants of the Crank--Nicolson scheme, which differ in the treatment of the nonlinear terms.

The structure of our paper is as follows: after introducing the model and its finite element discretization in Section~\ref{se:setting}, we describe different time stepping schemes in Section~\ref{sec:timeDisc}: Backward Euler \BD, \BDF, Crank--Nicolson (version based on midpoint rule \CNM or based on trapezoidal rule \CNT). Finally, in Section~\ref{sec:numExp}, we present several numerical experiments confirming the convergence and the stability proved in the previous section.

\section{Setting of the FSI problem}
\label{se:setting}
We consider a fluid-structure interaction problem consisting of a visco-elastic solid immersed in a fluid. The solid is initially distorted from its equilibrium configuration, so that it tends to return to its equilibrium position. In doing so, the region occupied by the fluid changes its shape, thus inducing a flow which in turn produces a force on the solid, which deforms accordingly. We assume that both the fluid and the solid are incompressible. An extension to our model to compressible solids has been studied in~\cite{BGH18} but is not going to be considered in this paper.

Let $\Omega\subset\RE^d$, with $d=2,3$, be a connected, open, and bounded domain with Lipschitz continuous boundary $\partial\Omega$. For simplicity, we assume that $\Omega$ is a polyhedron. The domain $\Omega$ is split into two non intersecting open domains $\Oft$ and $\Ost$ which represent the regions occupied at time $t$ by fluid and solid, respectively. Hence we have $\overline\Omega=\overline{\Oft}\cup\overline{\Ost}$. We denote by $\Gamma_t$ the interface between $\Oft$ and $\Ost$ and assume that it has empty intersection with the exterior boundary $\partial\Omega$.
Let $\B$ be the reference domain of $\Ost$, and let $\X:\B\to\Ost$ represent the corresponding deformation mapping. Hence a point $x\in\Ost$ is the image at time $t$ of a point $\s\in\B$, that is $\x=\X(\s,t)$. For simplicity, we assume that $\B=\Omega^s_0$ is the initial position of the solid. We denote by $\F=\Grad_s\X$ the deformation gradient and by $J=\det(\F)$ its Jacobian.
\begin{center}
\begin{figure}
\includegraphics[width=.8\textwidth]{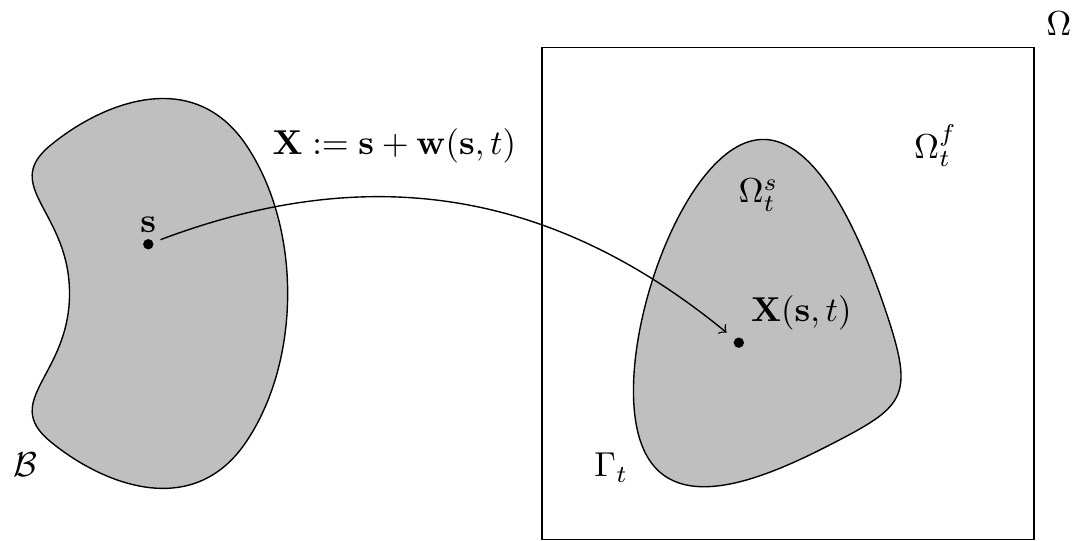}
\caption{Geometrical configuration of the FSI problem}
\end{figure}
\end{center}
We are going to use the following notation: $\u_f$, $p_f$, $\ssigma_f$, and $\rho_f$ denote, respectively, velocity, pressure, stress tensor and density in the fluid. We consider a Newtonian fluid characterized by the usual Navier--Stokes stress tensor

\begin{equation}
\label{eq:NSstresstensor}
\ssigma_f=-p_f\mathbf{I}+\nu_f\Grads\u_f,
\end{equation}
where $\nu_f$ is the fluid viscosity and $\Grads\u=(1/2)\left(\Grad\u_f+(\Grad\u_f)^\top\right)$.
In the fluid, we use an Eulerian description so that the material derivative is given by
\[
\dot\u_f=\frac{\partial\u_f}{\partial t}+\u_f\cdot\Grad\u_f.
\]
In the solid, $\u_s$, $p_s$, and $\rho_s$ stand, respectively, for velocity, pressure, and density. 
In the solid the Lagrangian framework is preferred, and the spatial description of the material velocity reads
\begin{equation}
    \label{eq:materialvel}
    \u_s(\x,t)=\frac{\partial\X(\s,t)}{\partial t}\Big|_{\x=\X(\s,t)}
\end{equation}
so that $\dot\u_s=\partial^2\X/\partial t^2$.
Moreover, we assume that the solid material is viscous-hyperelastic, so that the Cauchy stress tensor is given by the sum of a viscous part 
\begin{equation}
    \label{eq:viscous}
\ssigma_s^v=-p_s\mathbb{I}+\nu_s\Grads\u_s,
\end{equation}
where $\nu_s$ is the viscosity, and an elastic part $\ssigma_s^e$, which can be expressed in terms of the Piola--Kirchhoff stress tensor $\P$:
\begin{equation}
\label{eq:elastic}
    \P(\F(\s,t))=|\F(\s,t)|\ssigma_s^e(\x,t)\F^{-\top}(\s,t)\quad\text{for }\x=\X(\s,t).
\end{equation}
The Piola--Kirchhoff stress tensor is related to the positive energy density $W(\F)$, which characterizes hyperelastic materials, as follows:
\begin{equation}
\label{eq:PK}
(\P(\F(\s,t))_{\alpha i}=
\frac{\partial W}{\partial\F_{\alpha i}}(\F(\s,t))
=\left(\frac{\partial W }{\partial \F}(\F(\s,t))\right)_{\alpha i},
\end{equation}
where $i = 1,\ldots,m$ and $\alpha=1,\ldots,d$. The elastic potential energy of
the body is given by:
\begin{equation}
\label{eq:potenergy}
E\left(\X(t)\right)=\int_\B W(\F(s,t))\ds.
\end{equation}
Assuming that both the fluid and the solid material are incompressible, we have the following mathematical model for the fluid-structure system: 
\begin{equation}
\label{eq:model}
\aligned
&\rho_f\dot\u_f=\divergence\ssigma_f
&&\text{in }\Omega\setminus\B_t\\
&\divergence\u_f=0 &&\text{in }\Omega\setminus\B_t\\
&\rho_s\frac{\partial^2\X}{\partial t^2}
={\divergence} _s(|\F|\ssigma_s^v\F^{-\top}+\P(\F))\ 
&&\text{in }\B\\
&{\divergence}\u_s=0&&\text{in }\B_t\\
&\u_f=\frac{\partial\X}{\partial t}&&\text{on }\Gamma_t\\
&\ssigma_f\n_f=-(\ssigma_s^f+|\F|^{-1}\P\F^\top)\n_s
&&\text{on }\Gamma_t.
\endaligned
\end{equation}
The last two equations in~\eqref{eq:model} represent the transmission condition at the interface $\Gamma_t$.
The model is completed with initial and boundary conditions:
\begin{equation}
    \label{eq:init+bc}
\aligned
&\u_f(0)=\u_{f0} &&\text{in }\Omega_0^f\\
&\u_s(0)=\u_{s0} &&\text{in }\Omega_0^s\\
&\X(0)=\X_0 &&\text{in }\B\\
&\u_f(t)=0 &&\text{on }\partial\Omega.
\endaligned
\end{equation}
Following~\cite{BG17}, we apply a fictitious domain approach by extending the first equation in~\eqref{eq:model} to the whole domain $\Omega$ and by using the following new unknowns:
\begin{equation}
\label{eq:fictitious}
\u=\left\{
\begin{array}{ll}
\u_f&\text{ in } \Oft\\
\u_s&\text{ in } \Ost
\end{array}
\right.\quad
p=\left\{
\begin{array}{ll}
p_f&\text{ in } \Oft\\
p_s&\text{ in } \Ost.
\end{array}
\right.
\end{equation}
The extended velocity and pressure satisfy the following equation all over the domain:
\begin{equation}
\label{eq:NSwhole}
\aligned
&\rho_f\dot\u=\divergence(\nu\Grads\u)-\nabla p &&\quad\text{in }\Omega\\
&\divergence\u=0&&\quad\text{in }\Omega,
\endaligned
\end{equation}
where we set
\begin{equation}
\label{eq:nu}
\nu=\left\{\begin{array}{ll}
\nu_f&\text{ in } \Oft\\
\nu_s&\text{ in } \Ost.
\end{array}
\right.
\end{equation}
The first equation in~\eqref{eq:NSwhole} enforces that $\rho_f\dot\u_s=\divergence\ssigma_s^v$ in $\B_t$.
By taking into account~\eqref{eq:materialvel} and changing variable, this is equivalent to 
\[
\rho_f\frac{\partial^2\X}{\partial t^2}={\divergence} _s(|\F|\ssigma)s^v\F^{-\top}\quad \text{in }\B
\]
which subtracted from the third equation in~\eqref{eq:model} gives
\[
(\rho_s-\rho_f)\frac{\partial^2\X}{\partial t^2}={\divergence} _s\P(\F)\quad \text{in }\B
\]
and the model problem can be rewritten as follows: find $u$, $p$, $\X$ such that
\begin{equation}
\label{eq:modelfictitious}
\aligned
&\rho_f\dot\u=\divergence(\nu\Grads\u)-\nabla p &&\text{in }\Omega\\
&\divergence\u=0 &&\text{in }\Omega\\
&(\rho_s-\rho_f)\frac{\partial^2\X}{\partial t^2}={\divergence} _s\P(\F) &&\text{in }\B\\
&\u(\X(\s,t),t)=\frac{\partial\X(\s,t)}{\partial t}&&\text{on }\B\\
&\ssigma_f\n_f=-(\ssigma_s^v+|\F|^{-1}\P\F^\top)\n_s&&\text{on }\Gamma_t\\
&\u(0)=\u_{0} &&\text{in }\Omega\\
&\X(0)=\X_0 &&\text{in }\B\\
&\u(t)=0 &&\text{on }\partial\Omega.
\endaligned
\end{equation}

We observe that the second equation in~\eqref{eq:modelfictitious} enforces the divergence free constraint both for fluid and solid. In order to arrive to our weak formulation we multiply the first equation by a test function $v\in\Huo$ and integrate by parts the right hand side taking into account~\eqref{eq:fictitious} and the fact that the normal derivatives of $\u$ might jump across $\Gamma_t$, hence we have
\begin{equation}
\label{eq:previous}
\aligned
\int_\Omega \rho_f\dot\u\v\dx&=-\int_{\Omega}\nu\Grads\u:\Grads\v\dx+
\int_{\Omega}p\divergence\v\dx +\int_{\Gamma_t}\nu_f\Grads\u\,\n_f\v\d\gamma\\
&\quad-\int_{\Gamma_t}p\v\cdot\n_f\d\gamma
+\int_{\Gamma_t}\nu_s\Grads\u\,\n_s\v\d\gamma-\int_{\Gamma_t}p\v\cdot\n_s\d\gamma\\
&=-\int_{\Omega}\nu\Grads\u:\Grads\v\dx+\int_{\Omega}p\divergence\v\dx
+\int_{\Gamma_t}|\F|^{-1}\P(\F)\F^\top\,\n_s\v\d\gamma.
\endaligned
\end{equation}
On the other hand, multiplying by a test function $\Y\in\Hub$ the third equation in~\eqref{eq:modelfictitious} and integrating by parts, we obtain
\[
(\rho_s-\rho_f)\int_\B \frac{\partial^2\X}{\partial t^2}\Y\ds=
-\int_\B\P(\F):\grads\Y\ds+\int_{\partial\B}\P(\F)\,\mathbf{N}\Y\d\gamma,
\]
where $\mathbf{N}$ is the outward normal unit vector to $\partial\B$.
By change of variables, the integral on the boundary is equal to the last integral on the right hand side of~\eqref{eq:previous} if we choose $\v(\X(\s,t))=\Y(s)$ for $\s\in\partial\B$.

Let $\LL$ be a functional space to be defined later on and $\c:\LL\times\Hub\to\RE$ a continuous
bilinear form such that
\begin{equation}
\label{eq:cprop}
\c(\mu,\Y)=0\quad\forall\mu\in\LL \text{ implies }\Y=0.
\end{equation}
Possible definitions for $\LL$ and $\c$ are:
\begin{itemize}
\item $\c$ is the duality pairing between $\Hub$ its dual, that is $\LL=(\Hub)'$ and $\c(\mu,\Y)=\langle\mu,\Y\rangle_\B$ for $\mu\in\LL$, $\Y\in\Hub$;
\item $\c$ is the scalar product in $\Hub$, that is $\LL=\Hub$ and $\c(\mu,\Y)=(\mu,\Y)_\B+(\grads\mu,\grads\Y)_\B$ for $\mu,\Y\in\Hub$.
\end{itemize}
Then we introduce $\lambda\in\LL$ satisfying the following relation:
\begin{equation}
\label{eq:lambda}
\c(\lambda,\Y)=\int_{\partial\B}\P(\F)\,\mathbf{N}\Y\d\gamma \quad\forall\Y\in\Hub.
\end{equation}
Using the bilinear form $\c$ we can formulate the kinematic constraint in the fourth equation in~\eqref{eq:modelfictitious} as
\[
\c\left(\mu,\u(\X(\s,t),t)-\frac{\partial(\X(\s,t))}{\partial t}\right)=0
\]
and our problem can be rewritten in the following weak form (see, also, \cite{BCG15,BG17}).
\begin{problem}
\label{pb:pbvar}
For given $\u_0 \in\Huo$ and $\X_0\in W^{1, \infty}(\B)$, find $\u(t) \in\Huo$, $p(t) \in\Ldo$, $\X(t) \in\Hub$, and $\lambda(t) \in \LL$ such that for almost all $t \in (0, T)$:

\begin{equation}
	\label{eq:FSIvarDLM}
	\begin{aligned}
		&\rho_f \left(\derivative{}{t} \u(t),\v\right)_\Omega + b\left(\u(t), \u(t), \v\right) + a\left(\u(t), \v\right)  &&\\
		&\quad- \left(\divergence v, p(t) \right)_\Omega +\c\left(\lambda(t), \v(\X(\cdot, t))\right) = 0 & &\forall \v \in\Huo \\
		&\left( \divergence \u(t), q \right)_\Omega = 0  & &\forall q \in\Ldo
\\
		&\dr \left(\derivative{^2 \X}{t^2}(t), \Y \right)_\B +\left( \P(\F(t)),
\nabla_s \Y\right)_\B  -\c\left( \lambda(t), \Y\right) = 0  & &\forall \Y
\in\Hub \\
		&\c\left(\mu, \u(\X(\cdot, t), t)- \derivative{\X}{t}(t)\right) = 0 & &\forall \mu \in \LL \\
		&\u(x,0) = \u_0(x)  & &\mathrm{in }\ \Omega\\
		&\X(s,0) = \X_0(s) & &\mathrm{in }\ \B.\\
	\end{aligned}
\end{equation}
\end{problem}
In the above problem we have used the following notation: $\dr=\rho_s-\rho_f$, the scalar product in $L^2(D)$ is denoted by $(\cdot,\cdot)_D$, and
\begin{equation}
\label{eq:notation}
\aligned
& a(\u,\v)=(\nu\Grads\u,\Grads\v)_\Omega\\
&b(\u,\v,\w)=\frac{\rho_f}2\left((\u\cdot\Grad\v,\w)-(\u\cdot\Grad\w,\v)\right).
\endaligned
\end{equation}
\begin{remark}
We remark that the unknown $\lambda$ in Problem~\ref{pb:pbvar} plays the role of a Lagrange multiplier associated with the condition which enforces the kinematic constraint, that is the equality of the velocity $\u$ with the solid velocity in the region occupied by the structures, see also~\eqref{eq:materialvel}.
\end{remark}

\begin{remark}
We can introduce an alternative to the third equation in~\eqref{eq:FSIvarDLM} 
by splitting it into a system of two equations of first order in time, and
by introducing a new unknown $\dX(t)=\frac{\partial\X(t)}{\partial t}$:
\begin{equation}
\label{eq:splitting}
\aligned
&(\dX(t),\w)_\B=\left(\frac{\partial\X(t)}{\partial t},\w\right)_\B
&&\quad\forall\w\in L^2(\B)\\
&\left(\frac{\partial\dX(t)}{\partial t},\z\right)_\B+(\P(\F(t)),\grads\z)_\B
-\c(\lambda(t),\z)=0&&\quad\forall\z\in\Hub.
\endaligned
\end{equation}
This formulation is more suited when a second order time marching scheme is
used, since we do not need to introduce a second order approximation of the
second time derivative. 
\end{remark}

By choosing properly the test functions in~\eqref{eq:FSIvarDLM}, one can obtain the following energy estimate~\cite{BG17} :
\begin{proposition}
\label{pr:stab-cont}
Let us assume that $\dr\ge0$, that the potential energy density $W$ is a $C^1$ convex function
over the set of second order tensors, and that for almost every $t\in[0,T]$ the solution of Problem~\ref{pb:pbvar} is such that $\X(t)\in(W^{1,\infty}(\B))^d$
with $\frac{\partial \X}{\partial t}(t)\in L^2(\B)^d$,
then the following equality holds true
\begin{equation}
\label{eq:energyest}
\frac{\rho_f}2\frac{d}{dt}||\u(t)||^2_\Omega+\nu||\Grads\u(t)||^2_\Omega+
\frac{\dr}2\frac{d}{dt}\left\|\frac{\partial \X(t)}{\partial t}\right\|^2_\B
+\frac{d}{dt}E(\X(t))=0.
\end{equation}
\end{proposition}
In the above proposition $\|\cdot\|_D$ stands for the norm in $L^2(D)$.
We observe that the above proposition still holds true when the alternative
formulation in~\eqref{eq:splitting} is used.

\subsection{Finite element discretization}
\label{se:FEM}
Let $\T_h$ and $\T_h^\B$ be regular meshes in $\Omega$ and $\B$, respectively,
which are independent one from each other. The corresponding mesh sizes will be
denoted by $h_f$ and $h_s$, respectively. We consider two finite element spaces
$\V_h\subset\Huo$ and $Q_h\subset\Ldo$ such that the pair $(\V_h,Q_h)$ satisfies
the usual inf-sup condition for the Stokes equations. 
We assume that $\T_h^\B$ contains only simplices and introduce the space of continuous piecewise affine functions on $\T_h^\B$
\begin{equation}
\label{eq:Sh}
\S_h=\{\Y\in\Hub: \Y|_T\in\mathcal{P}^1(T)\ \forall T\in\T_h^\B\}.
\end{equation}
In order to discretize $\LL$, we set $\LL_h=\S_h$. With this definition we have that when $\LL=(\Hub)'$ and $\c$ is the duality pairing, we can compute easily $\c$ using the scalar product in $L^2(\B)$.

Then the discrete counterpart of Problem~\ref{pb:pbvar} reads as follows.
\begin{problem}
\label{pb:pbvarh}
For given $\u_{0h}\in\V_h$ and $\X_{0h}\in\S_h$, find $\u_h(t)\in\V_h$, $p_h(t)\in Q_h$, $\X_h(t)\in\S_h$, $\lambda_h(t)\in\LL_h$ such that for almost all $t \in (0, T)$:

\begin{equation}
	\label{eq:FSIvarDLMh}
	\begin{aligned}
		&\rho_f \left(\derivative{}{t}\u_h(t),\v_h\right)_\Omega + b\left(\u_h(t),\u_h(t),\v_h\right) + a\left(\u_h(t), \v_h\right)  &&\\
		&\quad- \left(\divergence\v_h, p_h(t) \right)_\Omega +\c\left(\lambda_h(t), \v_h(\X_h(\cdot, t))\right) = 0 & &\forall \v_h\in\V_h \\
		&\left(\divergence\u_h(t), q_h\right)_\Omega = 0  & &\forall q_h\in Q_h\\
		&\dr \left(\derivative{^2 \X_h}{t^2}(t), \Y_h\right)_\B +\left( \P(\F_h(t)), \grads\Y_h\right)_\B  -\c\left( \lambda_h(t), \Y_h\right) = 0  & &\forall \Y_h\in\S_h \\
		&\c\left(\mu_h, \u_h(\X_h(\cdot, t), t)- \derivative{\X_h}{t}(t)\right) = 0 & &\forall \mu_h\in\LL_h \\
		&\u_h(x,0) = \u_{0h}(x)  & &\mathrm{in }\  \Omega\\
		&\X_h(s,0) = \X_{0h}(s) & &\mathrm{in }\  \B.
	\end{aligned}
\end{equation}
\end{problem}
This semi discrete problem inherits the same energy estimate as the continuous
one, which can be proved with the same technique.

The system~\eqref{eq:splitting} can be discretized as follows: find
$\dX_h(t)\in\S_h$ and $\X_h(t)\in\S_h$ such that
\begin{equation}
\label{eq:splittingh}
\aligned
&(\dX_h(t),\w_h)_\B=\left(\frac{\partial\X_h(t)}{\partial t},\w_h\right)_\B
&&\quad\forall\w_h\in\S_h\\
&\left(\frac{\partial\dX_h(t)}{\partial t},\z_h\right)_\B
+(\P(\F_h(t)),\grads\z_h)_\B
-\c(\lambda_h(t),\z_h)=0&&\quad\forall\z_h\in\S_h.
\endaligned
\end{equation}
We observe that the first equation can be solved exactly. In the following we shall use these two equations instead of the third equation in~\eqref{eq:FSIvarDLMh} when we apply higher order time marching schemes.

\section{Higher-order time-stepping method}
\label{sec:timeDisc}
The system~\eqref{eq:FSIvarDLMh} is a system of ODEs with some algebraic
constraints in $\RE^n$, where $n = \dim(\V_h) + \dim(Q_h) + \dim(\S_h) +
\dim(\LL_h)$. The numerical solution of Problem~\ref{pb:pbvarh} can be computed by using some ODE/DAE solver. In order to avoid to use excessively small
time steps, in \cite{BCG15,BG17} the Backward Euler formula was applied for the
time-integration, and the unconditional stability of the scheme has been
proved. The aim of this paper is to introduce methods which achieve second
order convergence in time and are unconditionally stable. In particular, we
shall consider two one-step methods - based on the midpoint and trapezoidal
rules - and the \BDF method and analyze their stability. We observe that the
resulting fully discrete scheme is nonlinear, and we shall discuss how the
solution can be obtained.

\subsection{Backward Euler}
Before entering into the details of higher order methods, we recall the Backward Euler method analyzed in \cite{BCG15}. We subdivide the time interval $[0,T]$ into $N$ equal parts with size $\dt=T/N$ and subdivision points $t_n=n\dt$. Moreover, for a certain function $y(t)$, we set $y^n=y(t_n)$ and use the following finite difference in order to approximate the time derivatives:
\begin{equation}
\label{eq:Finitediff}
\aligned
&\frac{\partial y(t_{n+1})}{\partial t}\approx\frac{y^{n+1} -y^{n}}{\dt}\\
&\frac{\partial^2 y(t_{n+1})}{\partial t^2}\approx\frac{y^{n+1} -2y^{n}+y^{n-1}}{\dt^2}.
\endaligned
\end{equation}
Notice that both approximations are of first order.

The fully discrete version of Problem~\ref{pb:pbvar} using the Backward Euler scheme is the following one.
\begin{problem}
	Given $\u_{0h}\in\V_h$ and $\X_{0h}\in\S_h$, for all $n=1,\dots, N$ find $\u_h^n\in\V_h$, $p_h^n\in Q_h$, $\X_h^n\in\S_h$, and $\lambda_h^n\in\LL_h$ fulfilling:
	
	\begin{equation}
		\label{eq:FSIdiscBDF1}
		\begin{aligned}
			&\rho_f \left( \frac{\u_h^{n+1} - \u_h^{n}}{\dt},\v_h\right)_\Omega + b\left(\u_h^{n+1}, \u_h^{n+1}, \v_h\right)+ a\left(\u_h^{n+1}, \v_h\right)\\
			&\quad - \left( \divergence \v_h, p_h^{n+1}\right)_\Omega+
\c\left(\lambda_h^{n+1}, \v_h(\X_h^{n+1})\right) = 0 & & \forall \v_h \in\V_h \\
	&\left( \divergence \u_h^{n+1}, q_h \right)_\Omega = 0 
& & \forall q_h \in Q_h\\
&\delta\rho \left( \frac{\X_h^{n+1} - 2\X_h^{n} +
\X_h^{n-1}}{\dt^2}, \Y_h \right)_\B \\
&\quad+\left( \P(\F_h^{n+1}), \nabla_s \Y_h\right)_\B
- \c\left(\lambda_h^{n+1}, \Y_h\right) = 0& & \forall \Y_h \in\S_h \\
&\c\left(\mu_h, \u_h^{n+1}(\X_h^{n+1}) - \frac{\X_h^{n+1} -\X_h^n}{\dt}\right) = 0 & & \forall \mu_h \in \LL_h \\
			&\u_h^0 = \u_{0h}, \quad	\X_h^0 = \X_{0h}.
		\end{aligned}
	\end{equation}
\end{problem}

We recall the stability estimate proven in \cite{BCG15}.

\begin{proposition}
	Let the material behavior be governed by an energy density $W$ which is $\mathcal{C}^1$ and convex. Let $\u_h^n\in\V_h$ and $\X_h^n\in\S_h$, $ n = 1, \dots N$ be solutions of~\eqref{eq:FSIdiscBDF1}. Then the following estimate holds true:
	
	\begin{equation*}
		\begin{aligned}
			&\frac{\rho_f}{2 \dt} \left( \|\u^{n+1}_h\|_\Omega^2 - \|\u^n_h\|_\Omega^2\right)
			+ \nu \| \nabla_{sym} \u^{n+1}_h\|_\Omega^2  +\frac{E(\X^{n+1}_h) - E(\X^n_h)}{\dt}\\
			&+ \frac{\delta \rho}{2\dt} \left[ \left\|\frac{\X_h^{n+1} - \X_h^n}{\dt}\right\|_\B^2 - \left\|\frac{\X_h^n - \X_h^{n-1}}{\dt}\right\|_\B^2 \right]  \leq 0.
		\end{aligned}
	\end{equation*}
\end{proposition}

This system is highly nonlinear and hence not immediate to solve. It can be solved by a
fixed-point method or by replacing some quantities at time $t_{n+1}$ with their
known value at time $t_n$. For example, in~\cite{BCG15} the latter approach has
been preferred; the value of $\v(\X_h^{n+1})$ in the first equation has been
replaced by $\v(\X^n_h)$ and, similarly, $\u^{n+1}_h(\X^{n+1}_h)$ by
$\u^{n+1}_h(\X^n_h)$.
Moreover, as it is usual in the discretization of the Navier--Stokes equations,
the first argument in the trilinear form $b$ has been evaluated at time $t_n$.
Notice that the same stability estimate as for the fully implicit scheme has
been proved. We observe also that this modification introduces an approximation
of first order, which should not affect the accuracy of the Backward Euler
scheme, since it is of first order, too. 

At the first step $n=0$, the approximation of the second time derivative
requires the knowledge of $\X_h^{-1}$. Following~\cite{BG17},
we can compute this value using the kinematic constraint and the initial data 
\begin{equation*}
	\c\left(\mu_h, \u_h^0(\X_h^0) - \frac{\X_h^0 - \X_h^{-1}}{\dt}\right) = 0 \quad \forall \mu_h \in \LL_h.
\end{equation*}
\subsection{\BDF method}

Backwards differentiation formulae (BDF) are popular multistep methods to solve
stiff ODE problems. They can be derived by approximating the time derivative at
time $t_{n+1}$ with finite differences with order greater than or equal to 1.
In particular, using a first order finite difference we have the Backward Euler
method. In this paper we focus on the second order formula \BDF, hence, for a
certain function $y(t)$, we approximate the time derivative with  
\begin{equation}
\label{eq:BDF2}
	\frac{\partial y(t_{n+1})}{\partial t}\approx\frac{3 y^{n+1} - 4 y^{n} + y^{n-1}}{\dt}.
\end{equation}

When applied to ODEs, it is well-known that \BDF is convergent of order 2 and A-stable,
(see for more detail~\cite[Chapter 7]{DB08}). BDF methods have been successfully applied to the Navier--Stokes equations~\cite{IYD17}, to nonlinear structural mechanics~\cite{D10}, to electro-magnetic problems~\cite{OFI10}, and to Stokes--Darcy flow~\cite{CGSW12}. 

The application of the \BDF method to Problem~\ref{pb:pbvarh} gives the following formulation.
\begin{problem}
\label{pb:BDF2}
Given $\u_{0h} \in\V_h$, $\X_{0h} \in \S_h$, for all $n = 1, \dots, N$ find
$\u_h^n \in\V_h$, $p_h^n \in Q_h$, $\X_h^n \in\S_h$, $\dX_h^n\in\S_h$ and
$\lambda_h^n \in \LL_h$ fulfilling:
\begin{subequations}
\begin{alignat}{2}
&\rho_f \left( \frac{3\u_h^{n+1} - 4\u_h^{n}+\u_h^{n-1}}{2
\dt},\v_h\right)_\Omega + b\left(\u_h^{n+1}, \u_h^{n+1}, \v_h\right)\notag&&\\
&\qquad+ a\left(\u_h^{n+1}, \v_h\right)- \left( \divergence \v_h,
p_h^{n+1}\right)_\Omega\notag&&\\
&\qquad+ \c\left(\lambda_h^{n+1}, \v_h(\X_h^{n+1})\right) = 0 & & \forall
\v_h \in\V_h 
\label{eq:BDF21}\\
&\left( \divergence \u_h^{n+1}, q_h \right)_\Omega = 0  & &\forall q_h\in Q_h 
\label{eq:BDF22}\\
&(\dX_h^{n+1},\w_h)_\B=
\left(\frac{3\X_h^{n+1}-4\X_h^n+\X_h^{n-1}}{2\dt},\w_h\right)_\B
&&\forall\w_h\in\S_h
\label{eq:BDF23}\\
&\delta\rho \left( \frac{3\dX_h^{n+1}-4\dX_h^{n} +\dX_h^{n-1}}{2\dt},
\Y_h \right)_\B \notag&&\\
&\qquad +\left( \P(\F_h^{n+1}), \nabla_s \Y_h\right)_\B  - \c\left(
\lambda_h^{n+1}, \Y_h\right)= 0& & \forall \Y_h \in S_h 
\label{eq:BDF24}\\
&\c\left(\mu_h, \u_h^{n+1}(\X_h^{n+1}) - \frac{3\X_h^{n+1} -4 \X_h^n +
\X_h^{n-1}}{2 \dt}\right) = 0 & & \forall \mu_h \in \LL_h 
\label{eq:BDF25}\\
&\u_h^0 = \u_{0h},\quad		\X_h^0 = \X_{0h}.
\label{eq:BDF26}
\end{alignat}
\end{subequations}
\end{problem}
To start the computations, we additionally need the quantities $\u_h^1$ and
$\X_h^1$ which are usually obtained by a one-step method. To achieve second
order consistency theoretically, we need to apply a method of second order to
the first step. This can be done, for instance, by using the Crank--Nicolson scheme, which is
analyzed in this paper as well. In our numerical experiments we
also tried a start-up with the backward Euler method, which also produced
second order convergence.

In the following proposition, we prove an energy estimate for the scheme
described by Equations~\eqref{eq:BDF21}-\eqref{eq:BDF26}.

\begin{proposition}
\label{pr:BDF2stab}
Assume that $\delta \rho \geq 0$ and that the solids behavior is linear:
$\P(\F) = \kappa \F$.
Let $\u_h^n$, $\X_h^n$, $ n = 1, \dots N$, be solutions of Problem~\ref{pb:BDF2},
then the following estimate holds true:
\begin{equation}
\label{eq:BDF2stab}
\begin{aligned}
&\frac{\rho_f}{4\dt} \left[ \left\|\u_h^{n+1}\right\|_\Omega^2 + \left\|2
\u_h^{n+1} - \u_h^n\right\|_\Omega^2 - \left\|\u_h^n\right\|_\Omega^2 -
\left\|2 \u_h^n - \u_h^{n-1}\right\|_\Omega^2\right.\\
&\qquad	\left.	+ \left\|\u_h^{n+1}-2\u_h^n + \u_h^{n-1}\right\|_\Omega^2\right]
+ \nu \left\|\Grads \u_h^{n+1}\right\|_\Omega^2  \\
&\qquad	+ 
\frac{\dr}{4\dt^2}\left(\|\dX_h^{n+1}\|_\B^2+\|2\dX_h^{n+1}-\dX_h^n\|^2_\B\right.\\
&\qquad\left.
-\|\dX_h^n\|^2_\B-\|2\dX_h^n-\dX_h^{n-1}\|^2_\B
+\|\dX_h^{n+1}-2\dX_h^n+\dX_h^{n-1}\|^2_\B\right) \\
&\qquad	+
\frac{\kappa}{4\dt}
\left(\|\F_h^{n+1}\|_\B^2+\|2\F_h^{n+1}-\F_h^n\|^2_\B\right.\\
&\qquad\left.
-\|\F_h^n\|^2_\B-\|2\F_h^n-\F_h^{n-1}\|^2_\B
+\|\F_h^{n+1}-2\F_h^n+\F_h^{n-1}\|^2_\B\right)\le0.
		\end{aligned}
	\end{equation}
\end{proposition}
\begin{proof}
We mimic the proof of~\cite[Prop.~3]{BCG15} and use some useful tricks presented in~\cite{IYD17}.
We test~\eqref{eq:BDF21} with $\u_h^{n+1}$:
	
	\begin{equation*}
		\begin{aligned}
			&\rho_f \left( \frac{3\u_h^{n+1} - 4\u_h^{n}+\u_h^{n-1}}{2 \dt},\u_h^{n+1}\right)_\Omega + b\left(\u_h^{n+1}, \u_h^{n+1}, \u_h^{n+1}\right)  \\
			&\quad+ a\left(\u_h^{n+1}, \u_h^{n+1}\right) - \left( \divergence \u_h^{n+1}, p_h^{n+1}\right)_\Omega +\c\left(\lambda_h^{n+1}, \u_h^{n+1}(\X_h^{n+1})\right)  = 0.
		\end{aligned} 
	\end{equation*}
We remind that the trilinear form $b$ vanishes when the last two arguments
coincide. The third term reduces to 
$a\left(\u_h^{n+1}, \u_h^{n+1}\right) = \nu \left\| \Grads\u_h^{n+1}
\right\|_\Omega^2$. Then by testing~\eqref{eq:BDF22} with $p_h^{n+1}$,
we see that the divergence term vanishes.
Applying the following formula to the first term 
\begin{equation}
\label{eq:formula}
\frac{1}{2} (3a - 4b + c) a = \frac14\left(a^2 +(2a -b)^2 
- b^2 - (2b -c)^2 +(a - 2b + c)^2\right)
\end{equation}
	with $a = \u_h^{n+1}$, $b = \u_h^n$, $c = \u_h^{n-1}$, we obtain
\begin{equation*}
	\begin{aligned}
	&\left( \frac{3\u_h^{n+1} - 4\u_h^{n}+\u_h^{n-1}}{2 \dt},\u_h^{n+1}\right)_\Omega = 
	\frac{1}{4\dt} \left( \left\| \u_h^{n+1}\right\|_\Omega^2 + \left\| 2\u_h^{n+1} - \u_h^n\right\|_\Omega^2 \right.\\
	&\quad\left.-\left\| \u_h^n \right\|_\Omega^2 - \left\| 2\u_h^n - \u_h^{n-1}\right\|_\Omega^2 + \left\| \u_h^{n+1} - 2 \u_h^n + \u_h^{n-1}\right\|_\Omega^2 \right).
	\end{aligned}
\end{equation*}
We are now left with the treatment of the term involving $\c$. We take
$\mu_h=\lambda_h^{n+1}$ in~\eqref{eq:BDF25} and $\Y_h= \frac{3\X_h^{n+1} -
4\X_h^{n}+\X_h^{n-1}}{2 \dt}$ in~\eqref{eq:BDF24}, and use~\eqref{eq:BDF23}
to obtain:
\begin{equation*}
\begin{aligned}
\c\left(\lambda_h^{n+1}, \u_h^{n+1}(\X_h^{n+1})\right)
&=\c\left(\lambda_h^{n+1}, \frac{3\X_h^{n+1} - 4\X_h^{n}+\X_h^{n-1}}{2
\dt}\right)  \\
&=\delta \rho \left( \frac{3\dX_h^{n+1} - 4\dX_h^n +
\dX_h^{n-1}}{\dt ^2}, \dX_h^{n+1}\right)_\B
\\
&\quad\quad+ \left( \P(\F_h^{n+1}), \nabla_s \frac{3\X_h^{n+1} -
4\X_h^{n}+\X_h^{n-1}}{2 \dt}\right)_\B.\\
\end{aligned}
\end{equation*}
Using the formula~\eqref{eq:formula} we have
\[
\aligned
&\left( \frac{3\dX_h^{n+1} - 4\dX_h^n +
\dX_h^{n-1}}{\dt ^2},\dX_h^{n+1}\right)_\B
=\frac{1}{4\dt^2}\left(\|\dX_h^{n+1}\|_\B^2+\|2\dX_h^{n+1}-\dX_h^n\|^2_\B\right.\\&\quad\left.
-\|\dX_h^n\|^2_\B-\|2\dX_h^n-\dX_h^{n-1}\|^2_\B
+\|\dX_h^{n+1}-2\dX_h^n+\dX_h^{n-1}\|^2_\B\right).
\endaligned
\]
Similarly we bound the term involving the Piola--Kirchhoff
stress-tensor $\P$ using again~\eqref{eq:formula}:
\begin{equation*}
\begin{aligned}
&\left( \kappa \F_h^{n+1}, \frac{3\F_h^{n+1} 
- 4\F_h^{n}+\F_h^{n-1}}{2 \dt}\right)_\B = \frac{1}{4\dt}\left(\|\F_h^{n+1}\|_\B^2+\|2\F_h^{n+1}-\F_h^n\|^2_\B\right.\\&\quad\left.
-\|\F_h^n\|^2_\B-\|2\F_h^n-\F_h^{n-1}\|^2_\B
+\|\F_h^{n+1}-2\F_h^n+\F_h^{n-1}\|^2_\B\right).
\end{aligned}
\end{equation*}
Putting all these relations together yields the result and concludes our proof.
\end{proof}
Summing up $n=2,\dots,m-1\le N$ in~\eqref{eq:BDF2stab} and dropping out positive terms, yields
the following unconditional stability bound:
\begin{equation}
\label{eq:matrix}
\begin{aligned}
&\rho_f\left( \|\u_h^{m}\|_\Omega^2 +\| 2\u_h^{m} - \u_h^{m-1}\|_\Omega^2	
\right)
+ 4\nu\dt \sum_{n=1}^{m} \|\Grads \u_h^{n} \|_\Omega^2\\
&\qquad+\frac{\dr}{\dt}\left(\|\dX_h^{m}\|^2_\B+\|2\dX_h^{m}-\dX_h^{m-1}\|^2_\B\right)
+\kappa\left(\|\F_h^{m}\|^2_\B+\|2\F_h^{m}-\F_h^{m-1}\|^2_\B\right)\\
&\quad\leq
\rho_f\left(\|\u_h^1\|_\Omega^2+\|2\u_h^1-\u_h^{0}\|_\Omega^2\right) 
+ \frac{\dr}{\dt}\left(\|\dX_h^1\|_\Omega^2+\|2\dX_h^1-\dX_h^{0}\|_\Omega^2\right)\\
&\qquad +\kappa (\|\F_h^1\|_\Omega^2+\|2\F_h^1-\F_h^{0}\|_\Omega^2).
\end{aligned}
\end{equation}
We note that the scheme presented in~\eqref{eq:BDF21}-\eqref{eq:BDF26} is a fully implicit
system which involves the solution of a nonlinear equation, stemming from
the convection part in the Navier-Stokes equation as well as from the kinematic
coupling term. The problem can be presented in the following matrix form:
\begin{equation*}
	\begin{pmatrix}
		A(u_h^{n+1})   & -B^T & 0    & 0   				 & L_f(\X_h^{n+1})^T \\
		-B 		       & 0    & 0    & 0				 & 0			    \\
        0              & 0    & M_s  & -\frac3{2\dt}M_s  & 0	\\         
		0 		       & 0    & \frac{3\dr}{2\dt}M_s & A_s   & -L_s^T        \\
		L_f(\X_h^{n+1}) & 0   & 0  & -\frac{3}{2 \dt}L_s & 0		     	\\
	\end{pmatrix}	
	\begin{pmatrix}
		\u_h^{n+1} \\
		p_h^{n+1} \\
        \dX_h^{n+1}\\
		\X_h^{n+1} \\
		\lambda_h^{n+1} \\
	\end{pmatrix}
	= 
	\begin{pmatrix}
		g_1 \\
		0 \\
		g_2 \\
		g_3 \\
        g_4
	\end{pmatrix},
\end{equation*}
with
\begin{equation*}
	\begin{aligned}
		& A(\u_h^{n+1}) = \frac{3 \rho_f}{2\dt} M_f + K_f(\u_h^{n+1})\\
		& (M_f)_{ij} = \left( \phi_j,  \phi_i \right)_\Omega,
\quad(K_f(\u_h^{n+1}))_{ij} = a\left(\phi_j, \phi_i\right) + b\left(\u_h^{n+1}, \phi_j, \phi_i\right) \\
		& B_{ki} = \left( \divergence \phi_i, \psi_k\right)_\Omega\\
		& A_s = \frac{\delta \rho}{\dt^2} M_s + K_s,
\quad (M_s)_{ij} = \left( \chi_j, \chi_i\right)_\B,
\quad (K_s)_{ij} = \kappa \left( \nabla_s \chi_j, \nabla_s \chi_i \right)_\B\\
		& (L_f(\X_h^{n+1}))_{lj} = c\left(\zeta_l, \phi_j(\X_h^{n+1})\right), 
\quad (L_s)_{lj} = c\left(\zeta_l, \chi_j\right)\\
		& g_1 = \frac{2\rho_f}{\dt}M_f\u_h^n-\frac{\rho_f}{2\dt}M_f\u_h^{n-1},
        \quad g_2 = \frac1{\dt}M_s(4\dX_h^n-\dX_h^{n-1}\\
		& g_3 = \frac{\delta \rho}{\dt^2} M_s (2\X_h^n - \X_h^{n-1}), 
		\quad g_4 = -\frac{2}{\dt} L_s \X_h^n + \frac{1}{2\dt}L_s \X_h^{n-1}.
	\end{aligned}
\end{equation*}
Here $\phi_i$, $\psi_k$, $\chi_i$, and $\zeta_l$ denote the basis functions in
$\V_h$, $Q_h$, $\S_h$, and $\LL_h$, respectively.

Thanks to the theory developed in~\cite{BG17}, we know that the linearization
of the system above is associated to a steady saddle point problem which admits
a unique solution. Moreover, the finite element discretization is stable, thus
giving optimal convergence rates depending on the regularity of the solution.

We can either solve this system by a solver for nonlinear systems of equations
like a fixed point iteration or Newton like methods, or make this semi-implicit
by replacing the implicit terms with explicit ones. 
In the numerical experiments reported in the last section, we used a fixed point
iteration or the well known extrapolation formula $x^{n+1} = 2x^n - x^{n-1} + \mathcal{O}(\dt^2)$, in order to
replace $A(\u_h^{n+1})$ by $A(2\u_h^n - \u_h^{n-1})$ and $L_f(\X_h^{n+1})$ by
$L_f(2\X_h^n - \X_h^{n-1})$.




\subsection{Crank--Nicolson scheme}
The Crank--Nicolson scheme is another second order method widely used for the discretization of evolutionary equations thanks to the fact that it is A-stable and one-step.
In the literature, we can find two different schemes referred to as
Crank--Nicolson scheme, which can be obtained by applying either the midpoint (\CNM) or
the trapezoidal (\CNT) rule to integrate the ODE from $t_n$ to $t_{n+1}$.
Notice that the two methods coincide when the problem is linear.
As far as the Navier--Stokes equations are concerned,
in~\cite{HR90} the \CNM approach is used and the stability properties are analyzed, while
the \CNT formula has been introduced, for example, in~\cite[Chapter 7]{J16}. We are going to apply both formulations in order to compare their numerical performances.

Let us start with the \CNM version of the method.
\begin{problem}
\label{pb:FSIdiscCNM}
Given $\u_{0h} \in\V_h$ and $\X_{0h}\in\S_h$, for all $n=1,\dots,N$ find
$\u_h^n \in\V_h$, $p_h^n \in Q_h$, $\X_h^n \in \S_h$, $\dX_h^n\in\S_h$, and
$\lambda_h^n \in \LL_h$ such that
\begin{subequations}
\begin{alignat}{2}
&\rho_f \left( \frac{\u_h^{n+1} - \u_h^{n}}{\dt},\v_h\right)_\Omega 
+ b\left(\frac{\u^{n+1}_h+\u^n_h}{2},\frac{\u^{n+1}_h+\u^n_h}{2},\v_h\right)\notag \\
&\quad+ a\left(\frac{\u^{n+1}_h + \u^n_h}{2}, \v_h\right) 
		- \left( \divergence \v_h, p^{n+1} \right)_\Omega \notag\\
&\quad+ \c\left(\lambda_{n+1}, \v_h\left(\frac{\X^{n+1}_h
+ \X^n_h}{2}\right)\right)=0 && \forall \v_h \in \V_h 
\label{eq:CNM1}\\
&\left( \divergence \u_h^{n+1}, q_h \right)_\Omega = 0  &&\forall q_h \in Q_h 
\label{eq:CNM2}\\
&\left(\frac{\dX_h^{n+1}+\dX_h^n}2,\w_h\right)_\B=
\left(\frac{\X_h^{n+1}-\X_h^n}{\dt},\w_h\right)_\B &&\forall\w_h\in\S_h
\label{eq:CNM3}\\
&\delta\rho\left(\frac{\dX_h^{n+1}-\dX_h^n}{\dt},\Y_h\right)_\B 
+\left( \P(\F_h^{n+1}),\nabla_s \Y_h\right)_\B  - \c\left(\lambda_h^{n+1},
\Y_h\right) = 0\quad&&\forall \Y_h \in S_h
\label{eq:CNM4}\\
&\c\left(\mu_h, \frac{\u_h^{n+1} + \u_h^n}{2}\left(\frac{\X_h^{n+1} +
\X^n_h}{2}\right)  - \frac{\X_h^{n+1} - \X_h^n}{\dt}\right)  = 0&&
\forall \mu_h \in \Lambda_h \label{eq:CNM5}\\
&		\u_h^0 = \u_0 \quad \X_h^0 = \X_0.\label{eq:CNM6}
\end{alignat}
\end{subequations}
\end{problem}
Equation~\eqref{eq:CNM5} is obtained by applying the midpoint rule to the kinematic constraint. Hence we approximate
the value $\u\left(\X\left(t+\frac{\dt}2\right),t+\frac{\dt}2\right)$ and we average both
$\u$ and $\X$ providing an approximation which is second order accurate. 
The following proposition states a stability estimate similar to the one for the \BDF method.
\begin{proposition}
\label{pr:CNM}
Assume $\delta \rho \geq 0$ and that the material behavior is governed
by an energy density $W$ which is $\mathcal{C}^1$ and convex. Let $\u_h^n\in\V_h$ and
$\X_h^n\in\S_h$, for $n = 1, \dots N$, be the solutions of Problem~\ref{pb:FSIdiscCNM}. Then the following estimate holds true:
	\begin{equation*}
		\begin{aligned}
			&\frac{\rho_f}{2\dt} \left(\| \u_h^{n+1}\|_\Omega^2 - \|\u_h^{n}\|_\Omega^2\right)
			+\frac{\nu}{4} \| \Grads \u_h^{n+1} + \Grads \u_h^n \|_\Omega^2  \\
	&\qquad+\frac{\dr}{\dt^2}\left(\|\dX_h^{n+1}\|^2_\B-\|\dX_h^n\|^2_B\right)
	+\frac{E(\X^{n+1}_h) - E(\X^n_h)}{\dt} \leq 0.\\
		\end{aligned}		
	\end{equation*}
\end{proposition}
\begin{proof} We mimic again the proof of~\cite[Prop.~3]{BCG15}.
	We take $\v_h=\left(\u_h^{n+1} + \u_h^n\right)/2$ in~\eqref{eq:CNM1} to get:
	\begin{equation*}
		\begin{aligned}
			&\rho_f \left( \frac{\u_h^{n+1} - \u_h^{n}}{\dt},\frac{\u_h^{n+1} + \u_h^{n}}{2}\right)_\Omega \\
			&\qquad+ b\left(\frac{\u^{n+1}_h + \u^n_h}{2}, \frac{\u^{n+1}_h + \u^n_h}{2}, \frac{\u^{n+1}_h + \u^n_h}{2}\right) + a\left(\frac{\u^{n+1}_h + \u^n_h}{2}, \frac{\u^{n+1}_h + \u^n_h}{2}\right) \\
			&\qquad- \left( \divergence \frac{\u^{n+1}_h + \u^n_h}{2}, p_h^{n+1}\right)_\Omega 
			+ \c\left(\lambda_h^{n+1}, \frac{\u^{n+1}_h + \u^n_h}{2}\left(\frac{\X^{n+1}_h + \X^n_h}{2}\right) \right)  = 0. \\
		\end{aligned}
	\end{equation*}
	The trilinear form $b$ vanishes since the second two arguments coincide. 
	Equation~\eqref{eq:CNM2} implies that the terms including the divergence vanish. If the initial condition $\u^0_h$ is discretely divergence free, the theorem even hold for $n=0$.
	Hence the above equality reduces to
	\begin{equation*}
		\begin{aligned}
&			\frac{\rho_f}{2\dt} \left(  \|\u_h^{n+1}\|_\Omega^2 - \|\u_h^n\|_\Omega^2 \right)
			+ \frac{\nu}{4} \| \nabla_{sym}\u_h^{n+1} + \nabla_{sym}\u_h^{n}\|_\Omega^2 &\\ 
&\qquad	+\c\left(\lambda_h^{n+1}, \frac{\u^{n+1}_h + \u^n_h}{2}\left(\frac{\X^{n+1}_h + \X^n_h}{2}\right) \right) = 0. \\
		\end{aligned}
	\end{equation*}
We set $\mu_h=\lambda^{n+1}$ in~\eqref{eq:CNM5}, then, using the equations in the solid~\eqref{eq:CNM3} with $\Y_h=(\X_h^{n+1}-\X_h^n)/2$ and~\eqref{eq:CNM4} with $\w_h=(\dX_h^{n+1}-\dX_h^n)/2$, we arrive at 
\begin{equation*}
\begin{aligned}
\c\Bigg(\lambda_h^{n+1},& \frac{\u^{n+1}_h + \u^n_h}{2}\left(\frac{\X^{n+1}_h + \X^n_h}{2}\right) \Bigg) =\c\left(\lambda_h^{n+1}, \frac{\X_h^{n+1} - \X_h^{n}}{\dt}\right) \\
&= \delta \rho \left( \frac{\dX_h^{n+1}-\dX_h^n}{\dt},\frac{\X_h^{n+1}-\X_h^n}{\dt} \right)_\B +
			\left ( \P(\F_h^{n+1}), \nabla_s \frac{\X_h^{n+1}-\X_h^n}{\dt} \right )_\B  \\	
&=	{\dr}\left( \frac{\dX_h^{n+1}+\dX_h^n}{\dt},\frac{\dX_h^{n+1}-\dX_h^n}{\dt} \right)_\B
	+\left ( \P(\F_h^{n+1}), \frac{\F_h^{n+1}-\F_h^n}{\dt} \right )_\B   \\
&=\frac{\dr}{\dt^2}\left(\|\dX_h^{n+1}\|^2_\B-\|\dX_h^n\|^2_B\right)
+\left ( \P(\F_h^{n+1}), \frac{\F_h^{n+1}-\F_h^n}{\dt} \right )_\B.
\end{aligned}
	\end{equation*}
	Now we want to see how the last term relates to the energy~\eqref{eq:potenergy}. Let us define the function $\mathcal{W}:[0,1] \rightarrow \mathbb{R}$ by
	\begin{equation*}
		\mathcal{W}(t) = W(\F^n_h + t (\F^{n+1}_h - \F^n_h)).
	\end{equation*}
	The convexity of $\mathcal{W}$ is inherited from $W$, so we see $\mathcal{W}^\prime(1) \geq \mathcal{W}(1) - \mathcal{W}(0)$. Using the chain rule, we obtain $\mathcal{W}^\prime(1) = \P(\F^{n+1}_h): (\F^{n+1}_h - \F^n_h)$. This is exactly the integrand in the term we want estimate. We obtain
	\begin{equation*}
		\begin{aligned}
			\left ( \P(\F_h^{n+1}), \frac{\F_h^{n+1}-\F_h^n}{\dt} \right )_\B  &= 
			\frac{1}{\dt}\int_\B \mathcal{W}^\prime(1) ds \geq \\
			\qquad\frac{1}{\dt}\int_\B \mathcal{W}(1) - \mathcal{W}(0) ds &= 
			\frac{E(\X^{n+1}_h) - E(\X^n_h)}{\dt}.
		\end{aligned}
	\end{equation*}
	Putting all equations together yields the result.
\end{proof}
Again we can sum the equations with respect to $n$ from $0$ to $m-1\le N$ and get the unconditional stability bound:
\begin{equation*}
	\begin{aligned}
&		\frac{\rho_f}{2\dt} \|\u_h^m\|_\Omega^2
		+ \frac{\nu}{4} \sum_{n=1}^{m} \|\Grads (\u_h^{n} + \u_h^{n-1}) \|_\Omega^2
		+ \frac{\delta \rho}{\dt^2} \|\dX_h^{m}\|^2_\B
		+\frac{E(\X^N_h)}{\dt} \\
&\qquad\le\frac{\rho_f}{2\dt} \|\u_h^0\|_\Omega^2 + \frac{\delta \rho}{\dt^2} \|\dX_h^0\|^2_\B
		+\frac{E(\X_h^0)}{\dt}.\\
	\end{aligned}
\end{equation*}
If we use the \CNT method, we are led to consider the problem:
\begin{problem}
\label{pb:FSIdiscCNT}
Given $\u_{0h}\in\V_h$ and $\X_0 \in\S_h$, for all $n = 1, \dots, N$ find
$\u_h^n \in\V_h$, $p_h^n \in Q_h$, $\X_h^n \in S_h$, $\dX_h^n\in\S_h$ and
$\lambda_h^n \in \LL_h$, such that:
\begin{equation}
\label{FSIdiscCNT}
\begin{aligned}
&\rho_f \left( \frac{\u_h^{n+1} - \u_h^{n}}{\dt},\v_h\right)_\Omega 
+ \frac 1 2 b\left(\u^{n+1}_h, \u^{n+1}_h, \v_h\right) + \frac 12 b\left(\u_h^n, \u_h^n, \v\right) \\
&\quad
+ \frac{1}{2}a\left(\u^{n+1}_h+\u^{n}_h,\v_h\right) 
- \frac12\left( \divergence \v_h, p^{n+1}+p_h^n \right)_\Omega \\
&\quad
+ \frac 1 2\c\left(\lambda_h^{n+1}, \v_h\left(\X^{n+1}_h\right)\right)
+ \frac 1 2\c\left(\lambda_h^{n}, \v_h\left(\X^{n}_h\right)\right) = 0 
&& \forall \v_h \in\V_h \\
&\left( \divergence \u_h^{n+1}, q_h \right)_\Omega = 0  
&& \forall q \in Q_h \\
&\left(\frac{\dX_h^{n+1}+\dX_h^n}2,\w_h\right)_\B=
\left(\frac{\X_h^{n+1}-\X_h^n}{\dt},\w_h\right)_\B
&&\forall\w_h\in\S_h\\
&\dr\left(\frac{\dX_h^{n+1}-\dX_h^n}{\dt},\Y_h\right)+
\left(\frac{\P(\F_h^{n+1})+\P(\F_h^n)}2, \grads \Y_h\right)_\B\\  
&\quad- \c\left(\frac{\lambda_h^{n+1}+\lambda_h^n}{\dt}, \Y_h\right) = 0
&& \forall \Y_h \in S_h \\
&\c\left(\mu_h,\frac 1 2\u_h^{n+1}\left(\X_h^{n+1}\right) 
+ \frac 1 2\u_h^{n}\left(\X_h^{n}\right)-\frac{\X_h^{n+1} - \X_h^n}{\dt}\right)=0
&& \forall \mu_h \in \LL_h \\
&\u_h^0 = \u_{0h} \quad \X_h^0 = \X_{0h}. \\
\end{aligned}
\end{equation}
\end{problem}
We omit the stability analysis of this problem which is not straightforward not even for the Navier--Stokes equations. 
Nevertheless this scheme has been used in our numerical tests and instabilities seem not to occur.

Both Problems~\eqref{pb:FSIdiscCNM} and~\eqref{pb:FSIdiscCNT} can be presented in matrix form with the a structure similar
to that of \BDF, with some terms properly modified. The solution of the resulting nonlinear system can be obtained with a Newton like solver, Picard iterations, or by linearization. In our numerical experiments we adopted the latter two approaches.

\section{Numerical Experiments}
\label{sec:numExp}
In this section we present some numerical results, with the aim of verifying the accuracy of the higher-order schemes presented in the previous sections. We consider fluid and solid with the same density $\rho_f=\rho_s=\rho$ and same viscosity $\nu_f=\nu_s=\nu$.
In our numerical experiments we adopt either Picard iterations or the semi-implicit schemes obtained by linearization of the nonlinear terms with second order extrapolations. More precisely, we substitute the value of a certain quantity $y(t_{n+1})$ with $2y(t_n)-y(t_{n-1})$.

The first test case is the deformed annulus considered in~\cite{BCG15} where convergence tests illustrate the first order rate for the implicit Euler method. The second test case is the floating disk for which we want to analyze the volume conservation properties of our methods.

In all our experiments, we use triangular meshes both in $\Omega$ and $\B$, enhanced Bercovier--Pironneau elements (i.e.
$P1$iso$P2$/$(P1 + P0)$) for the $(\u, p)$ discretization (see~\cite{BCGG12}) and $P1$ elements for the discretization of the
body's deformation $\X$, body's velocity $\dX$ and the distributed Lagrange multipliers $\lambda$. We evaluate the bilinear form
$\c$ by means of the $L^2$ scalar product.

At each time step or at each fixed point iteration we have to assemble the matrix in~\eqref{eq:matrix} and to solve the
resulting algebraic system. First of all, we observe that the matrix is neither symmetric nor positive definite.
The sparsity pattern of the matrix is shown in Fig.~\ref{Sparsity}.
\begin{figure}
	\centering
	\includegraphics[width=0.7\textwidth]{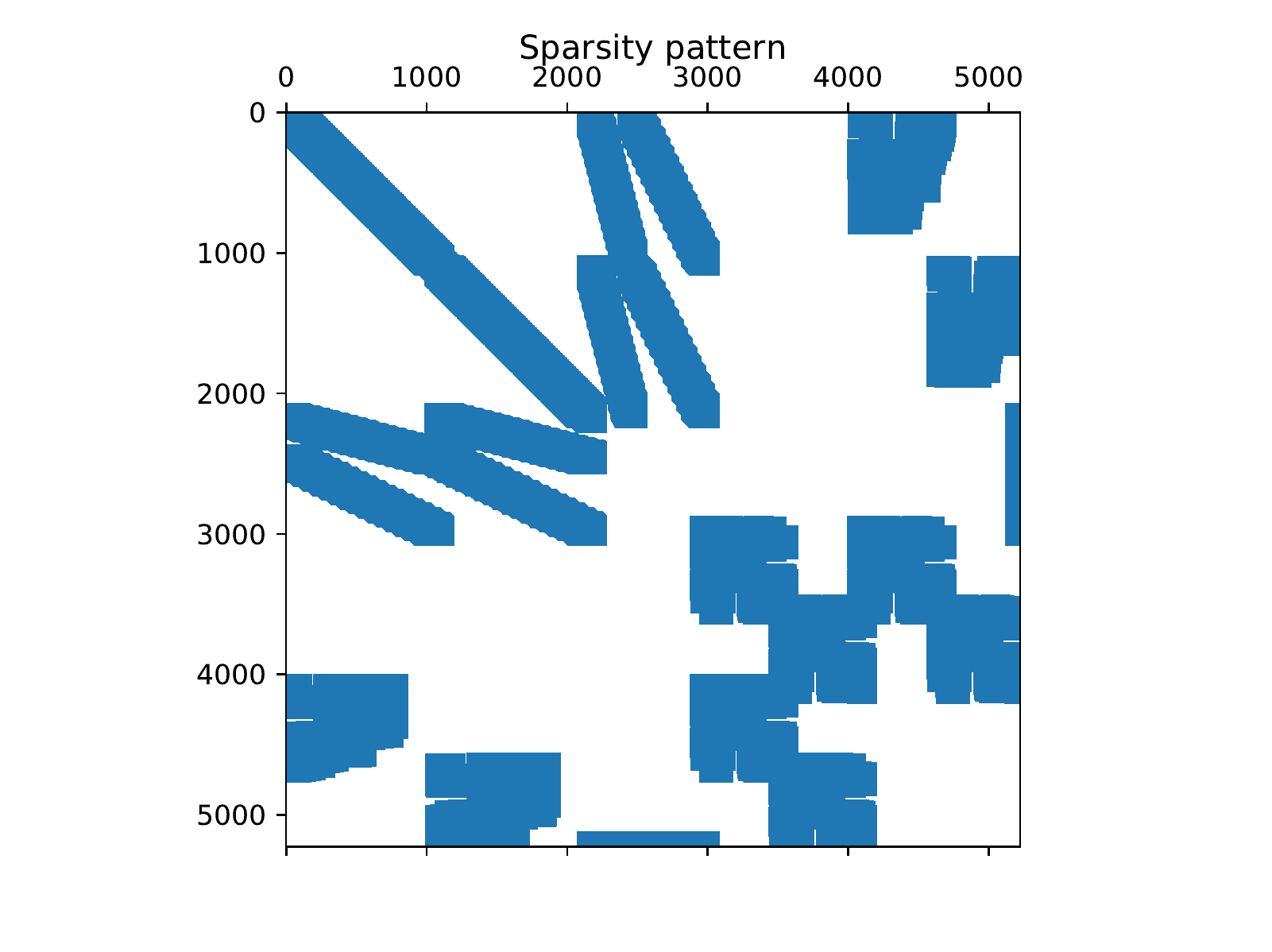}
	\caption{Sparsity pattern for a matrix arising from Equation
\eqref{eq:FSIdiscBDF1}}
	\label{Sparsity}
\end{figure} 
Another aspect that influences the computational time is the assembly of the coupling terms. 
In order to construct the contribution of the matrix $L_f(\X_h^n)$, we have to integrate on an element $T_s$ in the
Lagrangian mesh $\mathcal{T}_h^\B$ quantities living on the Eulerian mesh such as $\v(\X(\cdot, t),t)$. Numerically, this is done using some quadrature formula on $T_s$. Consequently, we have to find all intersecting elements $T_f \in \mathcal{T}_h$ of the fluid mesh such that $\X(T_s, t) \cap T_f \neq \emptyset$. Essentially this results in checking every deformed solid element against each fluid element. Bounding boxes are a good tool to detect when two cells surely do not intersect. This reduces the computational cost for each cell to cell comparison. Anyway, the computational cost for finding the intersecting cells grows linearly with the number of cells: 
\begin{equation*}
	\# \text{comparisons} = \# \text{fluid cells} \cdot \# \text{solid cells} \approx \frac{1}{h_f^d} \frac{1}{h_s^d}.
\end{equation*}
Although, the theoretical analysis allows to choose meshes for the fluid and the solid independently one from each other, 
it has been observed, for example in~\cite{HC12}, that the mesh parameter $h_s$ of the solid should be approximately half the size of the mesh parameter $h_f$ of the fluid. This can be explained as a trade-off between accuracy and computational effort. 

\subsection{The deformed annulus}
In this test, an annulus is placed at the center of a square filled with some fluid. Initially the annulus is deformed from its initial configuration and the fluid is at rest. The internal forces steer the annulus back into its undeformed configuration and set the surrounding fluid into motion. The material of the annulus is hyperelastic and described by the identity $\P(\F) = \kappa \F$. Thanks to the symmetry of the geometry,
we run the simulation in the upper right quarter of the domain. The fluid's domain is the unit square $\Omega = (0,1)^2$. On the upper and right boundary we impose no-slip boundary conditions $\u(\x,t) = 0$ for the fluid velocity.
The reference domain for the immersed structure is a section of the annulus: $\B = \{\x \in \mathbb{R}^2 : x_1, x_2 \geq 0, 0.3 \leq |\x| \leq 0.5 \}$. On the remaining part of the boundary, the fluid and the solid are allowed to move along the tangential direction, while the normal component is set to $0$.
The initial conditions are 
\begin{equation*}
	\u(\x,0) = 0,\quad\X(\s,0) = \begin{pmatrix} \frac{1}{1.4} s_1 \\ 1.4 s_2 \end{pmatrix}.
\end{equation*}
The meshes for the fluid and the structure are schematically reported in Fig.~\ref{AnnulusMeshes}. 
\begin{figure}
	\centering
	\includegraphics[height=0.4\textwidth]{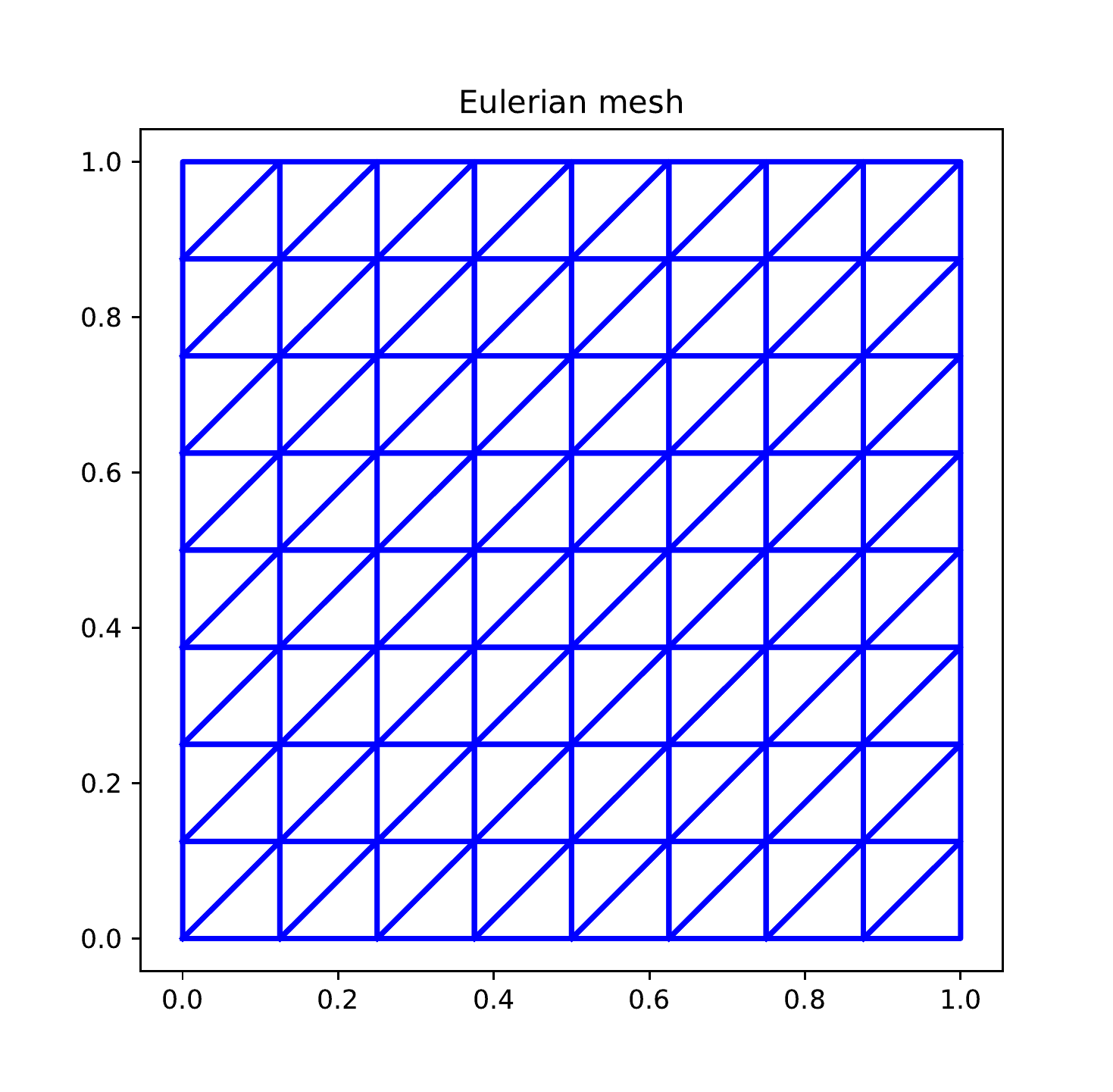}
	\includegraphics[height=0.4\textwidth]{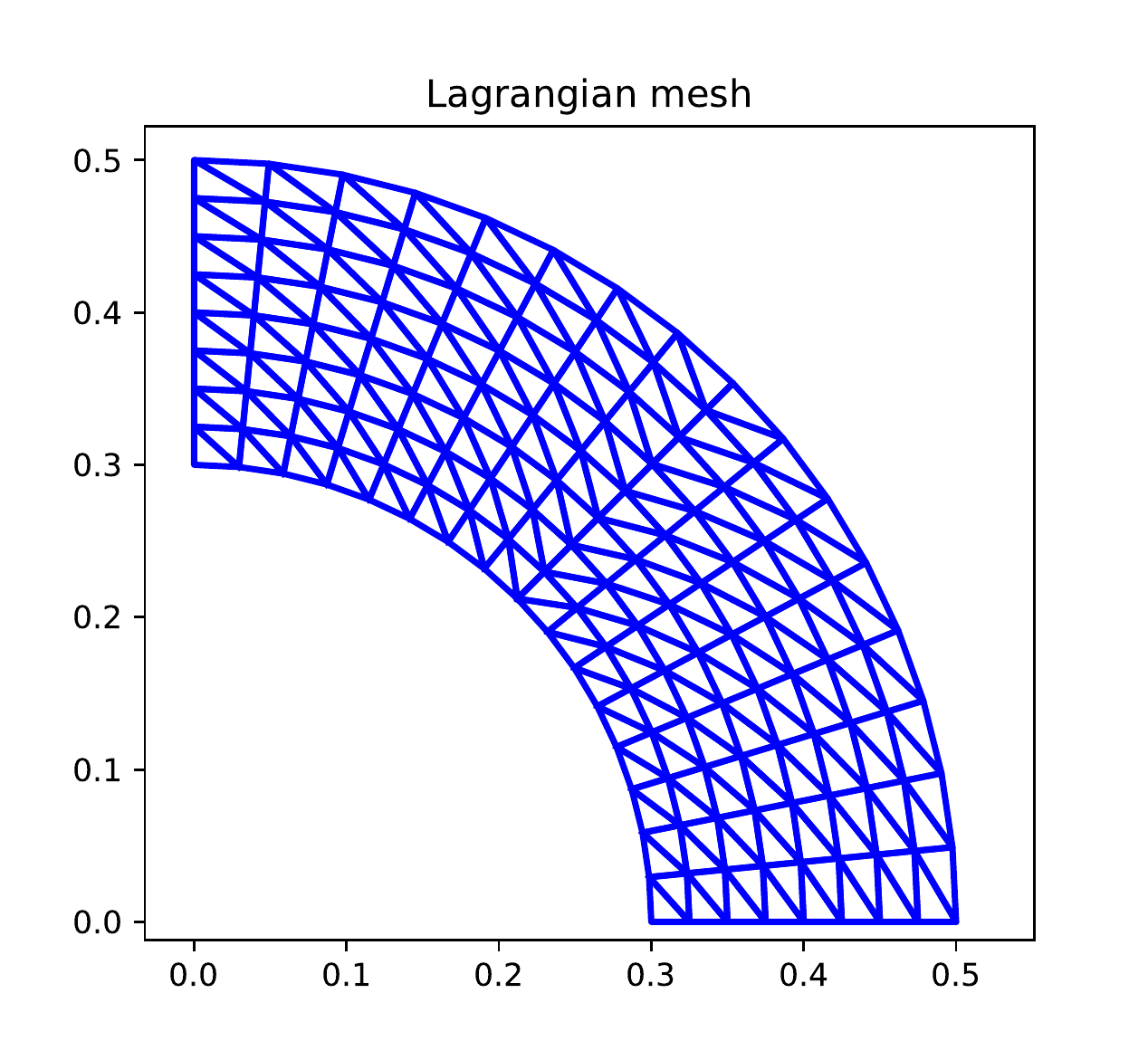}
	\caption{Meshes for the fluid and the structure}
	\label{AnnulusMeshes}
\end{figure} 
We use two meshes whose number of degrees of freedom can be found in Table \ref{TableMeshParams}. 
\begin{table}
	\centering
	\begin{tabular}{| c | c | c | c | c |}
		\hline
		 & DOFs $\u_h$ & DOFs $p_h$ & DOFs $\X_h$ & DOFs $\lambda_h$ \\
		\hline	
		coarse mesh (M = $8$)  &    $578$ &   $209$ &   $306$ &   $306$ \\
		fine mesh (M = $16$) &  $2,178$ &   $801$ & $1,122$ & $1,122$ \\
		\hline		
	\end{tabular}
	\caption{Mesh parameters}
	\label{TableMeshParams}
\end{table}
In Fig.~\ref{ExampleAnnulus} we report the position of the annulus and the streamlines of the velocity
corresponding to the following choice of parameters $\kappa = 10$, $\nu = 0.1$, $\rho = 1$, $T=1$. 
The \BDF method is used with $\dt = 0.05$. The snapshots are taken at $t = 0$, $t = 0.1$, $t = 0.5$, and $t=1$.
\begin{figure}
	\centering
	\includegraphics[width=0.4\textwidth]{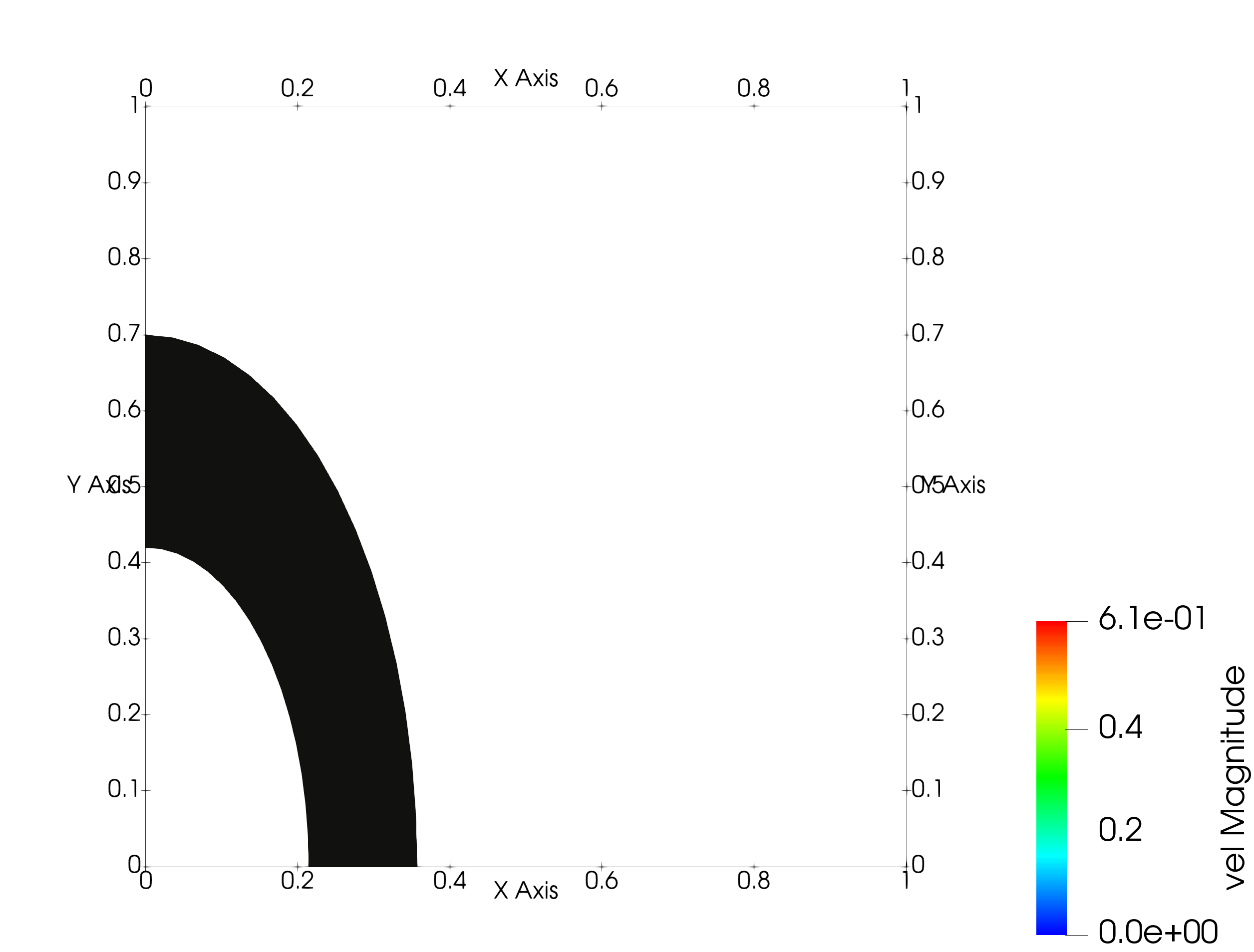}
	\includegraphics[width=0.4\textwidth]{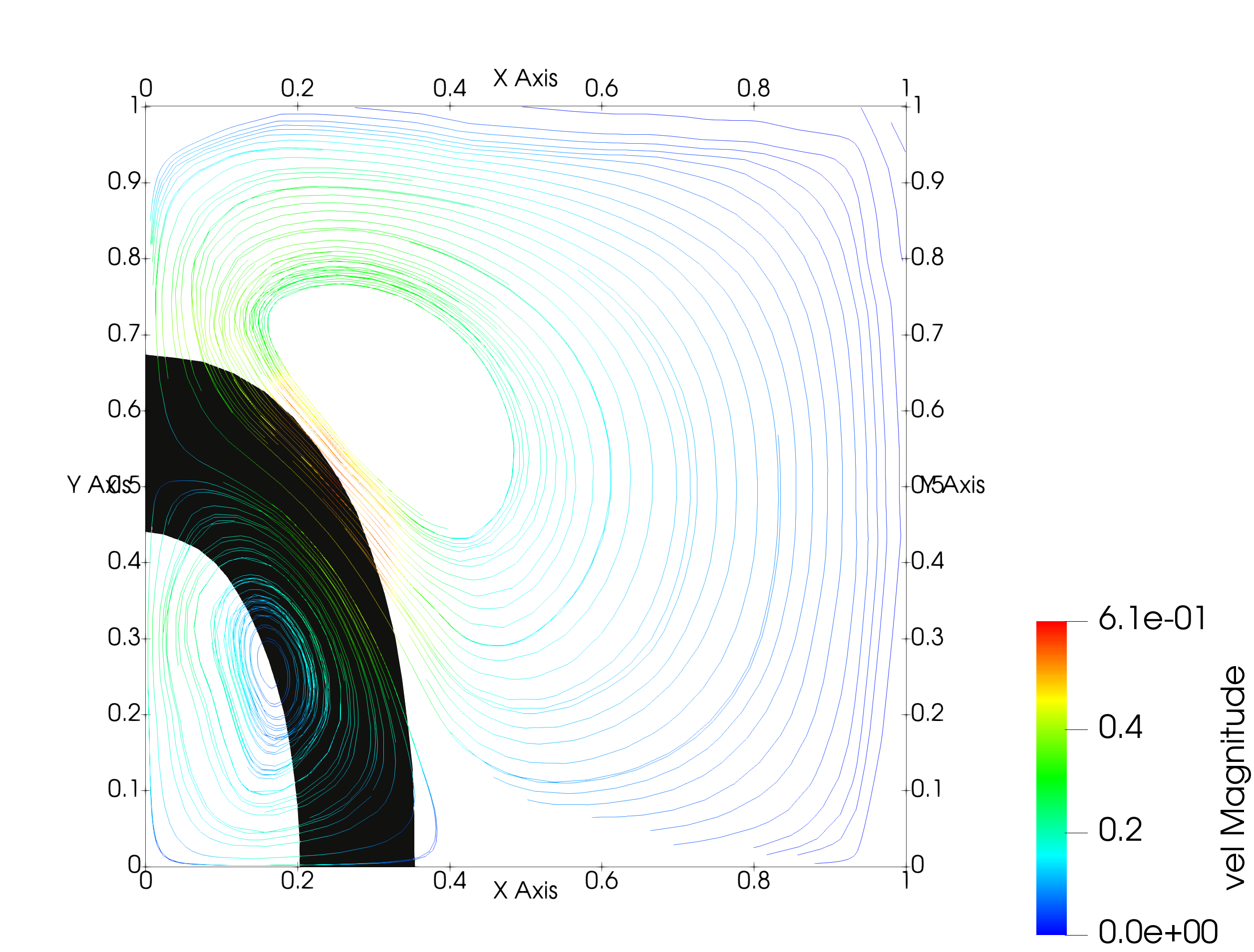}
	\includegraphics[width=0.4\textwidth]{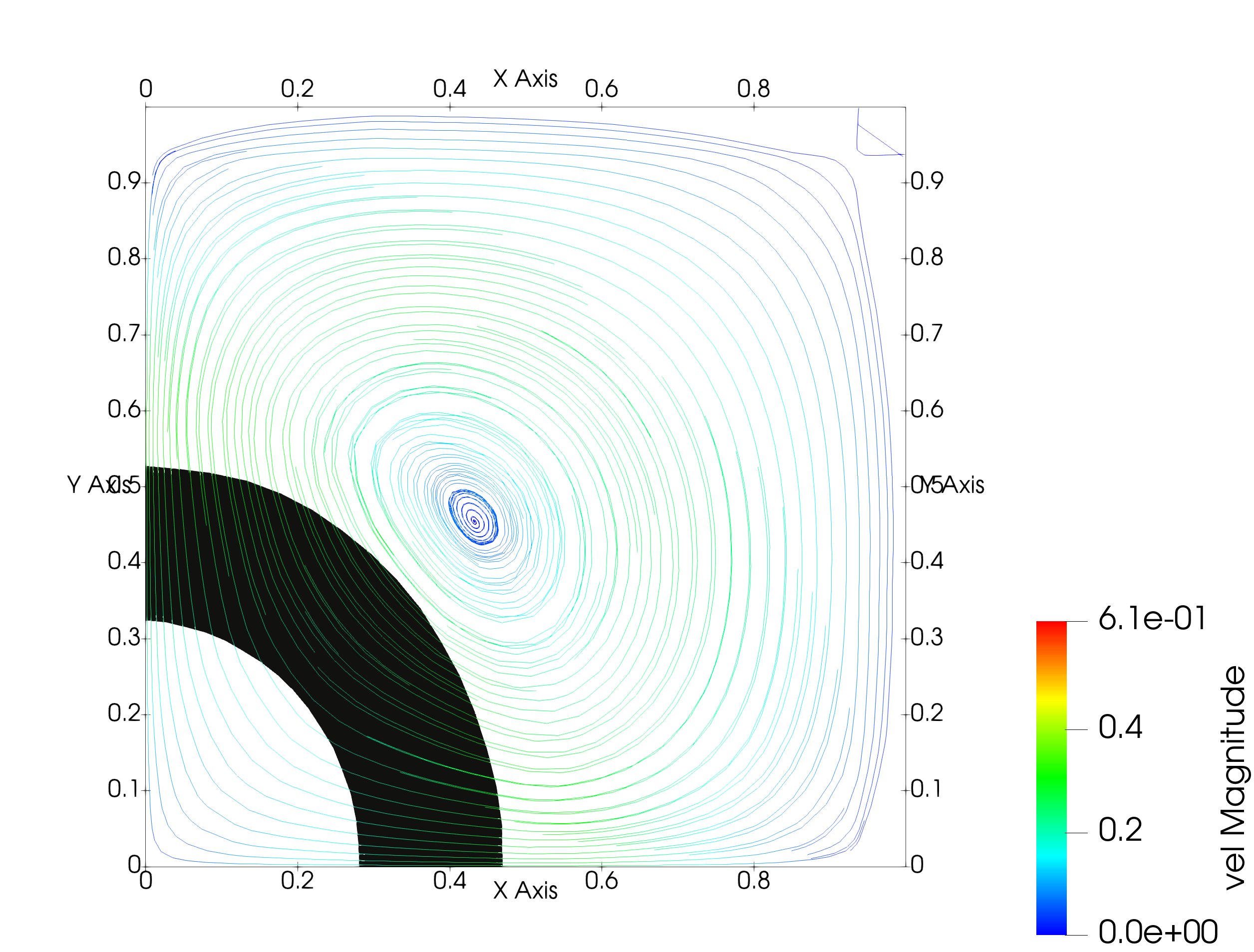}
	\includegraphics[width=0.4\textwidth]{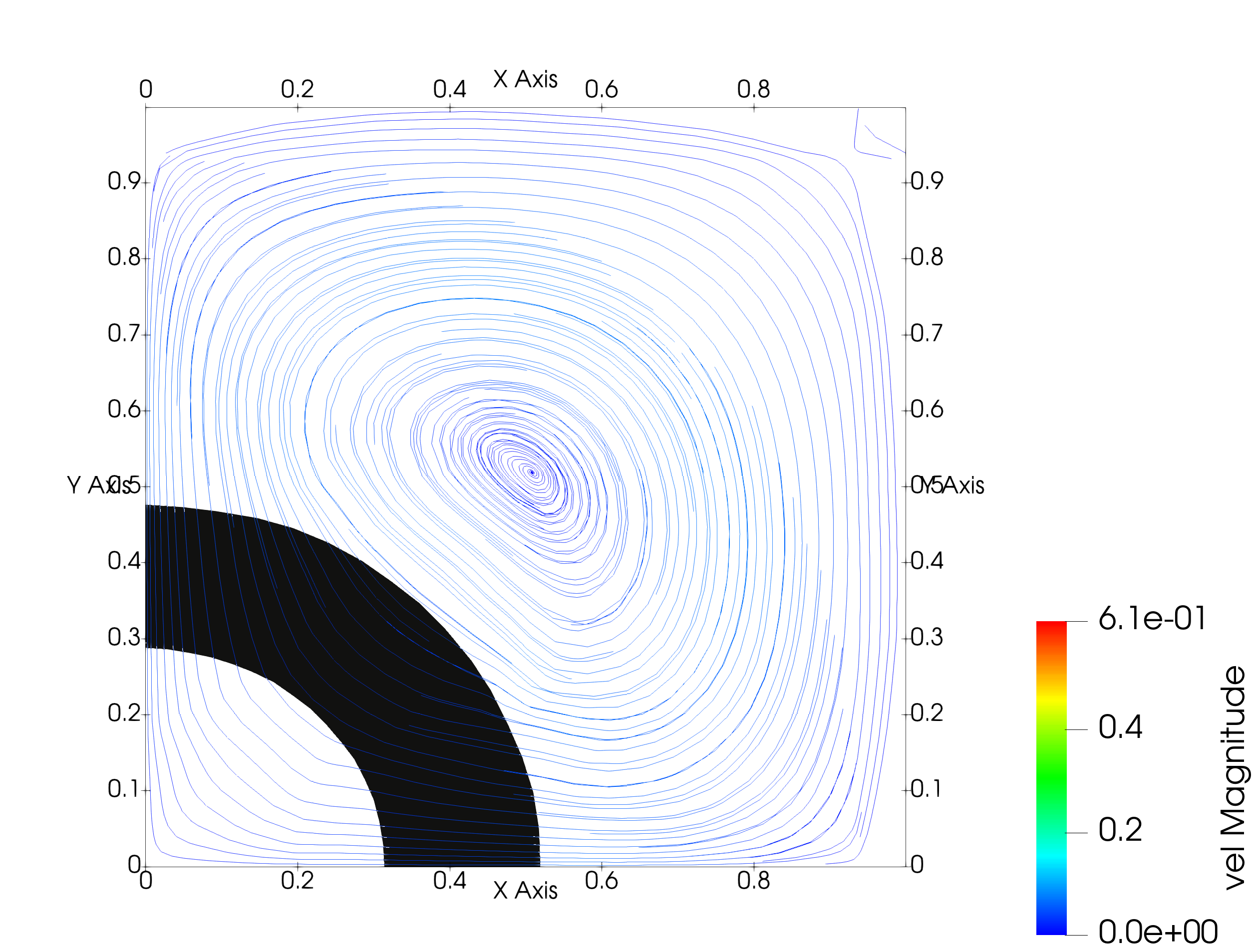}
	\caption[The deformed annulus]{The deformed annulus \footnotemark}
	\label{ExampleAnnulus}
\end{figure} 

In what follows, we report some tests to investigate the higher-order convergence of the \BDF method and the two version of the Crank--Nicolson scheme \CNM and \CNT and compare them with the backward Euler method \BD.
The convergence tests are done using the two meshes plotted in Fig.~\ref{AnnulusMeshes} and
parameters $\rho= 1$, $\nu = 1$, $\kappa = 10$, and $T=0.2$. Solutions have been calculated for $\dt = T/4,\ T/8,\ T/16,\ T/32$. Since in this case no analytic solution is available, for each mesh we computed a reference solution 
using the \BDF method with time-step $\dt = 0.001$ and we report the relative errors with respect to it.
When we use Picard iterations to obtain the solution of fully implicit schemes, we calculate the residual
in the $L^2$-norm and stop the iteration when it is less than the tolerance $\epsilon=10^{-6}$.

We first investigate the convergence of the fully implicit scheme. The relative errors in the $L^2$ norm are displayed in Tables \ref{TableConvNavierImpl8} and \ref{TableConvNavierImpl16}.

\begin{table}
	\centering
	\begin{tabular}{| l | l r | l r | l r | l r |}
		\multicolumn{9}{c}{Fluid velocity} \\
		\hline
		& \multicolumn{2}{|c|}{\BD} & \multicolumn{2}{|c|}{\BDF} & \multicolumn{2}{|c|}{\CNM} & \multicolumn{2}{|c|}{\CNT}\\
		$\dt$ & $L^2$ error & rate & $L^2$ error & rate & $L^2$ error & rate & $L^2$ error & rate\\
		\hline	
		$0.05$   &$7.63\cdot 10^{-2}$&      &$2.97\cdot 10^{-2}$&      &$2.37\cdot 10^{-1}$&      &$2.42\cdot 10^{-1}$& \\
		$0.025$  &$4.11\cdot 10^{-2}$&$0.89$&$4.90\cdot 10^{-3}$&$2.60$&$6.24\cdot 10^{-2}$&$1.92$&$6.02\cdot 10^{-2}$&$2.00$\\
		$0.0125$ &$2.13\cdot 10^{-2}$&$0.95$&$1.13\cdot 10^{-3}$&$2.11$&$1.21\cdot 10^{-2}$&$2.36$&$1.10\cdot 10^{-2}$&$2.45$\\
		$0.00625$&$1.08\cdot 10^{-2}$&$0.97$&$2.86\cdot 10^{-4}$&$1.98$&$2.03\cdot 10^{-3}$&$2.58$&$9.95\cdot 10^{-4}$&$3.47$\\
		\hline
	\end{tabular}

	\begin{tabular}{| l | l r | l r | l r | l r |}
		\multicolumn{9}{c}{Structure deformation} \\
		\hline
		& \multicolumn{2}{|c|}{\BD} & \multicolumn{2}{|c|}{\BDF} & \multicolumn{2}{|c|}{\CNM} & \multicolumn{2}{|c|}{\CNT}\\
		$\dt$ & $L^2$ error & rate & $L^2$ error & rate & $L^2$ error & rate & $L^2$ error & rate\\
		\hline
		$0.05$   &$1.60\cdot 10^{-3}$&      &$4.39\cdot 10^{-4}$&      &$1.43\cdot 10^{-3}$&      &$2.95\cdot 10^{-4}$& \\
		$0.025$  &$8.40\cdot 10^{-4}$&$0.93$&$9.75\cdot 10^{-5}$&$2.17$&$7.90\cdot 10^{-4}$&$0.86$&$6.89\cdot 10^{-5}$&$2.10$\\
		$0.0125$ &$4.29\cdot 10^{-4}$&$0.97$&$2.53\cdot 10^{-5}$&$1.95$&$4.06\cdot 10^{-4}$&$0.96$&$7.53\cdot 10^{-6}$&$3.19$\\
		$0.00625$&$2.17\cdot 10^{-4}$&$0.98$&$6.41\cdot 10^{-6}$&$1.98$&$2.04\cdot 10^{-4}$&$0.99$&$2.05\cdot 10^{-6}$&$1.88$\\
		\hline
	\end{tabular}
	\caption{Convergence results for the fully implicit scheme on the coarse mesh}
	\label{TableConvNavierImpl8}
\end{table}

\begin{table}
	\centering
	\begin{tabular}{| l | l r | l r | l r | l r |}
		\multicolumn{9}{c}{Fluid velocity} \\
		\hline
		& \multicolumn{2}{|c|}{\BD} & \multicolumn{2}{|c|}{\BDF} & \multicolumn{2}{|c|}{\CNM} & \multicolumn{2}{|c|}{\CNT}\\
		$\dt$ & $L^2$ error & rate & $L^2$ error & rate & $L^2$ error & rate & $L^2$ error & rate\\
		\hline	
		$0.05$   &$9.05\cdot 10^{-2}$&      &$3.62\cdot 10^{-2}$&      &$2.28\cdot 10^{-1}$&      &$2.26\cdot 10^{-1}$& \\
		$0.025$  &$4.87\cdot 10^{-2}$&$0.89$&$5.05\cdot 10^{-3}$&$2.84$&$6.23\cdot 10^{-2}$&$1.87$&$6.04\cdot 10^{-2}$&$1.91$\\
		$0.0125$ &$2.54\cdot 10^{-2}$&$0.94$&$1.20\cdot 10^{-3}$&$2.07$&$2.28\cdot 10^{-2}$&$1.45$&$2.07\cdot 10^{-2}$&$1.54$\\
		$0.00625$&$1.29\cdot 10^{-2}$&$0.98$&$3.53\cdot 10^{-4}$&$1.77$&$5.27\cdot 10^{-3}$&$2.11$&$4.03\cdot 10^{-3}$&$2.36$\\
		\hline
	\end{tabular}

	\begin{tabular}{| l | l r | l r | l r | l r |}
		\multicolumn{9}{c}{Structure deformation} \\
		\hline
		& \multicolumn{2}{|c|}{\BD} & \multicolumn{2}{|c|}{\BDF} & \multicolumn{2}{|c|}{\CNM} & \multicolumn{2}{|c|}{\CNT}\\
		$\dt$ & $L^2$ error & rate & $L^2$ error & rate & $L^2$ error & rate & $L^2$ error & rate\\
		\hline
		$0.05$   &$1.98\cdot 10^{-3}$&      &$5.19\cdot 10^{-4}$&      &$1.65\cdot 10^{-3}$&      &$4.04\cdot 10^{-4}$& \\
		$0.025$  &$1.05\cdot 10^{-3}$&$0.92$&$9.79\cdot 10^{-5}$&$2.41$&$9.27\cdot 10^{-4}$&$0.84$&$8.48\cdot 10^{-5}$&$2.25$\\
		$0.0125$ &$5.31\cdot 10^{-4}$&$0.99$&$3.13\cdot 10^{-5}$&$1.64$&$4.90\cdot 10^{-4}$&$0.92$&$2.47\cdot 10^{-5}$&$1.78$\\
		$0.00625$&$2.70\cdot 10^{-4}$&$0.98$&$1.35\cdot 10^{-5}$&$1.22$&$2.50\cdot 10^{-4}$&$0.97$&$3.47\cdot 10^{-6}$&$2.83$\\
		\hline
	\end{tabular}
	\caption{Convergence results for the fully implicit scheme on the fine mesh}
	\label{TableConvNavierImpl16}
\end{table}

The backward Euler method shows a clean first order rate of convergence, while the \BDF method gives a clean second order rate of convergence only on the coarse mesh and in the fine mesh for the fluid velocity. 
The results for the structure displacement are not very clear on the fine mesh. The second order rate of 
convergence is achieved only for larger value of the time step and deteriorates as the time step decreases.
On the other hand, we see that in this case the error is about $10^{-5}$, which is likely close to the accuracy that our fine mesh can provide. 
The \CNT method achieves second order convergence for fluid and structure on both meshes, while the \CNM scheme provides a second order convergence only for the fluid velocity. In general, the \CNT scheme seems to perform better from the point of view of the rate of convergence and of the size of the error.

In order to investigate the behavior of the nonlinear solver,
we track the maximum number of Picard iterations needed to reach the given tolerance. The results can be found in Tables \ref{TableNavierNonlinConv8} and \ref{TableNavierNonlinConv16}.
We observe that smaller time steps give smaller number of iterates. This can be motivated by the fact that we use the solution at the previous time step as the initial value for the iteration, therefore for smaller time step it is closer to the value at the current time.
\begin{table}
	\centering
	\begin{tabular}{| l | c | c | c | c |}
		\hline
		$\dt$ & \BD & \BDF & \CNM & \CNT \\
		\hline
		$0.05$	 &$5$&$5$&$4$&$7$\\	
		$0.025$	 &$4$&$4$&$3$&$4$\\	
		$0.0125$ &$3$&$3$&$3$&$3$\\	
		$0.00625$&$3$&$3$&$2$&$3$\\	
		\hline		
	\end{tabular}
	\caption{Maximum iterates of the nonlinear solver on the coarse mesh}
	\label{TableNavierNonlinConv8}
\end{table}

\begin{table}
	\centering
	\begin{tabular}{| l | c | c | c | c |}
		\hline
		$\dt$ & \BD & \BDF & \CNM & \CNT \\
		\hline
		$0.05$	 &$10$&$5$&$6$&$6$\\
		$0.025$	 &$6$ &$5$&$5$&$4$\\
		$0.0125$ &$6$ &$4$&$4$&$4$\\
		$0.00625$&$4$ &$4$&$3$&$3$\\
		\hline
	\end{tabular}
	\caption{Maximum iterates of the nonlinear solver on the fine mesh}
	\label{TableNavierNonlinConv16}
\end{table}
In the case of fluid structure interaction, not only the nonlinear convection term has to be assembled at each iteration, but also the one coupling fluid and structure. Therefore, the semi-implicit scheme is attractive as it requires the solution of only one big system at each time steps. Results for the convergence study of the semi-implicit versions of the methods can be found in Tables \ref{TableConvNavierSemiimpl8} and \ref{TableConvNavierSemiimpl16}.
\begin{table}
	\centering
	\begin{tabular}{| l | l r | l r | l r | l r |}
		\multicolumn{9}{c}{Fluid velocity} \\
		\hline
		& \multicolumn{2}{|c|}{\BD} & \multicolumn{2}{|c|}{\BDF} & \multicolumn{2}{|c|}{\CNM} & \multicolumn{2}{|c|}{\CNT}\\
		$\dt$ & $L^2$ error & rate & $L^2$ error & rate & $L^2$ error & rate & $L^2$ error & rate\\
		\hline	
		$0.05$   &$7.78\cdot 10^{-2}$&      &$3.05\cdot 10^{-2}$&      &$2.49\cdot 10^{-1}$&      &$2.58\cdot 10^{-1}$& \\
		$0.025$  &$4.17\cdot 10^{-2}$&$0.90$&$7.89\cdot 10^{-3}$&$1.95$&$6.24\cdot 10^{-2}$&$2.00$&$6.74\cdot 10^{-2}$&$1.94$\\
		$0.0125$ &$2.17\cdot 10^{-2}$&$0.95$&$3.14\cdot 10^{-3}$&$1.33$&$1.24\cdot 10^{-2}$&$2.33$&$2.64\cdot 10^{-2}$&$1.35$\\
		$0.00625$&$1.10\cdot 10^{-2}$&$0.97$&$1.29\cdot 10^{-3}$&$1.29$&$2.25\cdot 10^{-3}$&$2.47$&$3.18\cdot 10^{-3}$&$3.06$\\
		\hline
	\end{tabular}

	\begin{tabular}{| l | l r | l r | l r | l r |}
		\multicolumn{9}{c}{Structure deformation} \\
		\hline
		& \multicolumn{2}{|c|}{\BD} & \multicolumn{2}{|c|}{\BDF} & \multicolumn{2}{|c|}{\CNM} & \multicolumn{2}{|c|}{\CNT}\\
		$\dt$ & $L^2$ error & rate & $L^2$ error & rate & $L^2$ error & rate & $L^2$ error & rate\\
		\hline
		$0.05$   &$1.67\cdot 10^{-3}$&      &$6.79\cdot 10^{-4}$&      &$1.44\cdot 10^{-3}$&      &$3.52\cdot 10^{-4}$& \\
		$0.025$  &$8.65\cdot 10^{-4}$&$0.95$&$2.70\cdot 10^{-4}$&$1.33$&$7.91\cdot 10^{-4}$&$0.86$&$2.60\cdot 10^{-4}$&$0.44$\\
		$0.0125$ &$4.36\cdot 10^{-4}$&$0.99$&$1.24\cdot 10^{-4}$&$1.12$&$4.05\cdot 10^{-4}$&$0.97$&$1.53\cdot 10^{-5}$&$4.08$\\
		$0.00625$&$2.17\cdot 10^{-4}$&$1.01$&$5.71\cdot 10^{-5}$&$1.12$&$2.05\cdot 10^{-4}$&$0.98$&$9.43\cdot 10^{-6}$&$0.70$\\
		\hline
	\end{tabular}
	\caption{Convergence results for the semi-implicit scheme on the coarse mesh}
	\label{TableConvNavierSemiimpl8}
\end{table}
We see a similar behavior as the one given by Picard iterations. The backward Euler method gives a good first order convergence. The \BDF method shows a higher convergence order, which slows down for smaller time-steps in the fluid domain, in particular it is almost first order for the structure deformation.  
The behavior of the \CNM is not very clear, in fact only the error in the fluid velocity on the coarse mesh achieves second order convergence. The \CNT method provides better results which seem to respect the second order of convergence both for fluid velocity and structure deformation independently of the mesh.

\begin{table}
	\centering
	\begin{tabular}{| l | l r | l r | l r | l r |}
		\multicolumn{9}{c}{Fluid velocity} \\
		\hline
		& \multicolumn{2}{|c|}{\BD} & \multicolumn{2}{|c|}{\BDF} & \multicolumn{2}{|c|}{\CNM} & \multicolumn{2}{|c|}{\CNT}\\
		$\dt$ & $L^2$ error & rate & $L^2$ error & rate & $L^2$ error & rate & $L^2$ error & rate\\
		\hline	
		$0.05$   &$9.18\cdot 10^{-2}$&      &$3.89\cdot 10^{-2}$&      &$2.36\cdot 10^{-1}$&      &$2.39\cdot 10^{-1}$& \\
		$0.025$  &$5.05\cdot 10^{-2}$&$0.86$&$8.59\cdot 10^{-3}$&$2.18$&$7.54\cdot 10^{-2}$&$1.64$&$7.06\cdot 10^{-2}$&$1.76$\\
		$0.0125$ &$2.63\cdot 10^{-2}$&$0.94$&$3.32\cdot 10^{-3}$&$1.37$&$4.24\cdot 10^{-2}$&$0.83$&$2.22\cdot 10^{-2}$&$1.67$\\
		$0.00625$&$1.33\cdot 10^{-2}$&$0.98$&$1.40\cdot 10^{-3}$&$1.24$&$2.19\cdot 10^{-2}$&$0.96$&$4.19\cdot 10^{-3}$&$2.40$\\
		\hline
	\end{tabular}

	\begin{tabular}{| l | l r | l r | l r | l r |}
		\multicolumn{9}{c}{Structure deformation} \\
		\hline
		& \multicolumn{2}{|c|}{\BD} & \multicolumn{2}{|c|}{\BDF} & \multicolumn{2}{|c|}{\CNM} & \multicolumn{2}{|c|}{\CNT}\\
		$\dt$ & $L^2$ error & rate & $L^2$ error & rate & $L^2$ error & rate & $L^2$ error & rate\\
		\hline
		$0.05$   &$2.03\cdot 10^{-3}$&      &$7.86\cdot 10^{-4}$&      &$1.81\cdot 10^{-3}$&      &$6.51\cdot 10^{-4}$& \\
		$0.025$  &$1.06\cdot 10^{-3}$&$0.93$&$3.28\cdot 10^{-4}$&$1.26$&$9.75\cdot 10^{-4}$&$0.89$&$1.31\cdot 10^{-4}$&$2.31$\\
		$0.0125$ &$5.34\cdot 10^{-4}$&$1.00$&$1.44\cdot 10^{-4}$&$1.18$&$5.10\cdot 10^{-4}$&$0.93$&$4.82\cdot 10^{-5}$&$1.44$\\
		$0.00625$&$2.69\cdot 10^{-4}$&$0.99$&$6.31\cdot 10^{-5}$&$1.19$&$2.55\cdot 10^{-4}$&$1.00$&$1.29\cdot 10^{-5}$&$1.90$\\
		\hline
	\end{tabular}
	\caption{Convergence results for the semi-implicit scheme on the fine mesh}
	\label{TableConvNavierSemiimpl16}
\end{table}
Tables~\ref{TableNavierSemiimplRes8} and~\ref{TableNavierSemiimplRes16} report the residual obtained with the semi-implicit scheme which is higher of two orders of magnitude than the tolerance prescribed in the fixed point iterations. However, this fact does not influence too much the size of the errors.

\begin{table}
	\centering
	\begin{tabular}{| l | c | c | c | c |}
		\hline
		$\dt$ & \BD & \BDF & \CNM & \CNT \\
		\hline
		$0.05$   &$3.83\cdot 10^{-3}$&$3.83\cdot 10^{-3}$&$9.87\cdot 10^{-3}$&$9.64\cdot 10^{-2}$\\
		$0.025$	 &$2.09\cdot 10^{-3}$&$2.30\cdot 10^{-3}$&$1.24\cdot 10^{-3}$&$2.17\cdot 10^{-2}$\\
		$0.0125$ &$7.41\cdot 10^{-4}$&$8.26\cdot 10^{-4}$&$3.62\cdot 10^{-4}$&$7.90\cdot 10^{-3}$\\
		$0.00625$&$2.23\cdot 10^{-4}$&$2.45\cdot 10^{-4}$&$1.08\cdot 10^{-4}$&$9.55\cdot 10^{-4}$\\	
		\hline
	\end{tabular}
	\caption{Maximum residual in the semi-implicit scheme on the coarse mesh}
	\label{TableNavierSemiimplRes8}
\end{table}

\begin{table}
	\centering
	\begin{tabular}{| l | c | c | c | c |}
		\hline
		$\dt$ & \BD & \BDF & \CNM & \CNT \\
		\hline
		$0.05$   &$3.59\cdot 10^{-2}$&$3.59\cdot 10^{-2}$&$3.43\cdot 10^{-2}$&$3.81\cdot 10^{-2}$\\
		$0.025$  &$1.03\cdot 10^{-2}$&$1.07\cdot 10^{-2}$&$8.19\cdot 10^{-3}$&$1.02\cdot 10^{-2}$\\
		$0.0125$ &$5.28\cdot 10^{-3}$&$5.87\cdot 10^{-3}$&$1.54\cdot 10^{-3}$&$2.27\cdot 10^{-3}$\\
		$0.00625$&$1.46\cdot 10^{-3}$&$1.44\cdot 10^{-3}$&$4.33\cdot 10^{-4}$&$5.08\cdot 10^{-4}$\\	
		\hline
	\end{tabular}
	\caption{Maximum residual in the semi-implicit scheme on the fine mesh}
	\label{TableNavierSemiimplRes16}
\end{table}
\subsection{The floating disk}
In \cite{G12} it is pointed out that the Immersed Boundary Method may show poor volume conservation properties. Roy, Heltai, and Costanzo in~\cite{RHC15} have compared the volume preservation properties of their numerical method and propose it as a benchmark problem for fluid-structure interaction codes. In the following we investigate the volume preserving properties of the higher-order time-stepping schemes applied to a circular disk placed in a lid-driven cavity. The movement of the fluid induces a motion to the disk, which is being deformed and transported. 

The computational domain is a square with unit length: $\Omega = (0,1)^2$. The disk has a diameter of $0.2$ and its center is initially placed at $(0.6, 0.5)$. We describe Dirichlet boundary conditions on the fluid velocity: on the left, right, and bottom part of the boundary no-slip boundary conditions are enforced; on the upper part the fluid is set to $\u = (1,0)^\top$. The fluid and the solid have the same density $\rho_f = \rho_s = 1$. The viscosity is $\nu = 0.01$ and the material behaves as $\P(\F) = \kappa \F$ with $\kappa = 0.1$. The final time of the simulation is $T=4$. For the space discretization the same code as in the previous test is used on finer meshes with $18,818$ DOFs for the fluid velocity and $7,009$ DOFs for the pressure. The dimension of the space $\S_h$ is equal to
$4,402$. The maximal diameter is $h_f = 0.029$ for fluid elements, and $h_s = 0.012$ for structure cells. 

\begin{figure}
 	\centering
 	\includegraphics[width=0.7\textwidth]{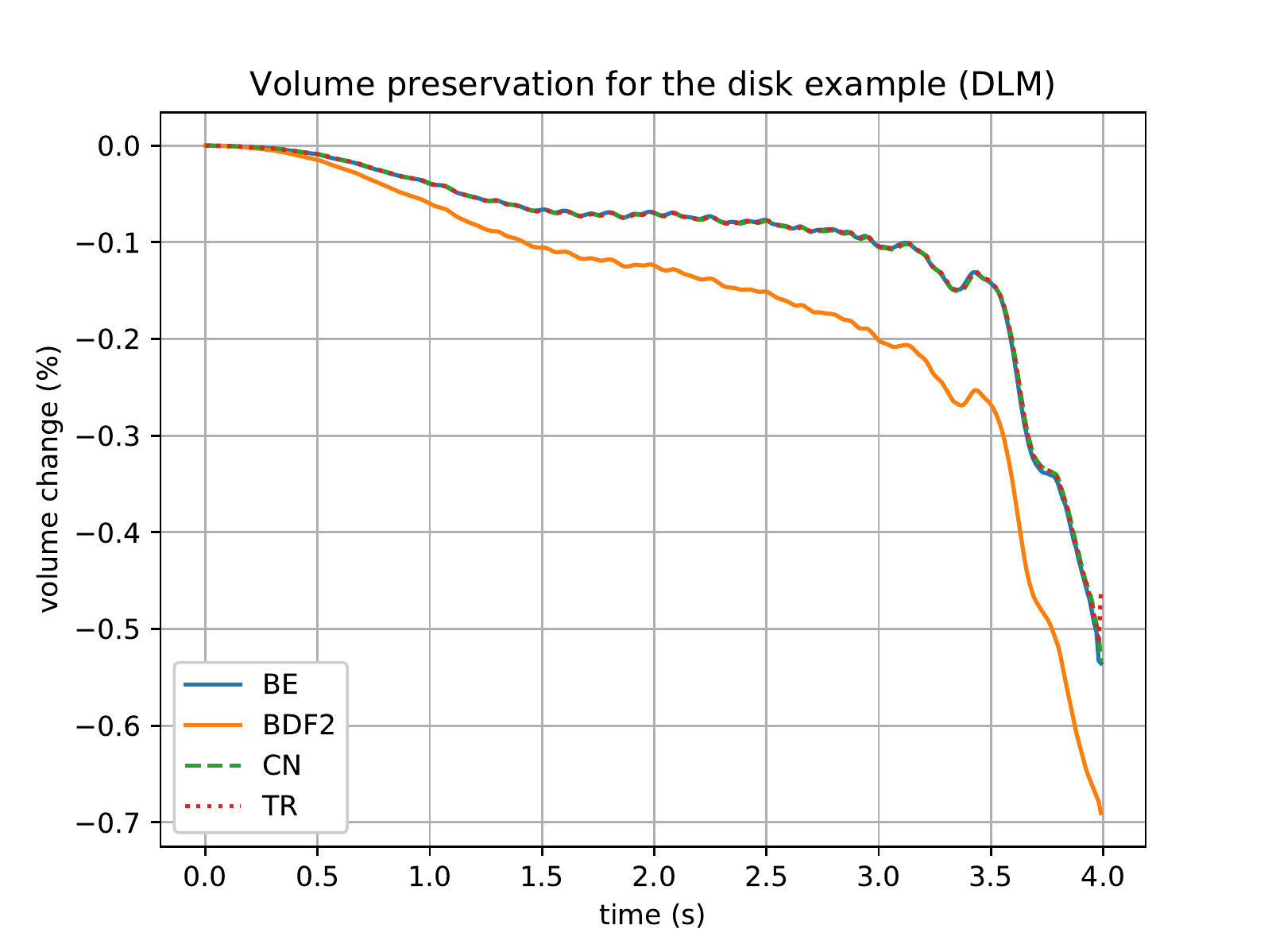}
 	\caption{Volume preservation over time}
 	\label{VolumePreservationDLM}
 \end{figure} 

The simulation has been run for all the four time marching schemes with a time-step of $\dt = 0.01$. At each time-step the volume of the immersed solid $\B_t$ is calculated and compared to the volume of the non-deformed solid. The percentage of volume change is plotted over time in Fig.~\ref{VolumePreservationDLM}.
Throughout the literature on Immersed Boundary Method this problem has been addressed.
Griffith and Luo~\cite{GL12} use a combination of finite differences for the fluid part and finite elements for the structure part. They report a volume conservation, which is more or less equal to results reported here. Wang and Zhan~\cite{WZ09} use finite element discretizations for fluid and solid, while the coupling is enforced via some interpolation function, which mimics the Dirac $\delta$ in the continuous problem. They report a much larger volume change than what we have found. Even their volume preserving scheme performs worse. In our numerical experiments, we note that the one step methods give better volume preservation, whereas \BDF method tends to reduce more the volume of the immersed solid.

The \BD, the \CNM, and the \CNT schemes produce the same volume preservation pattern. In the beginning the volume decreases, from $t=1.5$ to $t=3.0$ the volume almost stays constant. In the last part of the time interval, the volume decreases again  faster. At the end the volume is decreased by $0.5\%$. The \BDF method produces the same pattern, but the volume change is larger, at the final time the volume is decreased by $0.7\%$. It is interesting to note that the \BDF method, which is more accurate than the Backward Euler, produces a bigger volume change.

To investigate further the influence of the discretization parameters $\dt$, $h_f$, and $h_s$ on the volume preservation, the equations have been solved again, once on the same fine mesh with the doubled time-step $\dt = 0.02$ and once on a coarser mesh with the same time-step $\dt = 0.01$. The coarse mesh consists of $8,450$ velocity DOFs, $3,137$ pressure DOFs and $1,986$ DOFs for the structure deformation and the Lagrangian multiplier.
The volume preservation is plotted over time in Fig.~\ref{VolumePreservationCoarse} for the coarse mesh (left) and the coarser time-step (right). One can observe that the qualitative behavior is similar to that reported in Fig.~\ref{VolumePreservationDLM} for the fine mesh and double time step $\dt=0.02$. The behavior on the coarser mesh, reported in Fig.~\ref{VolumePreservationCoarse} (left), presents a larger kink at the end of the interval with respect to that in Fig.~\ref{VolumePreservationDLM} and the absolute volume change is approximately doubled, at the end the volume is decreased by $-1.0\%$. 

We conclude that in absolute numbers the volume change is the same for both simulations with less accuracy. The spatial discretization has a bigger influence on the qualitative behavior than the time discretization. This results is perfectly compatible with the fact that the area loss is strictly related to the approximation of the divergence free condition which, in turn, depends on the spatial discretization.

 \begin{figure}
	\centering
	\includegraphics[width=0.4\textwidth]{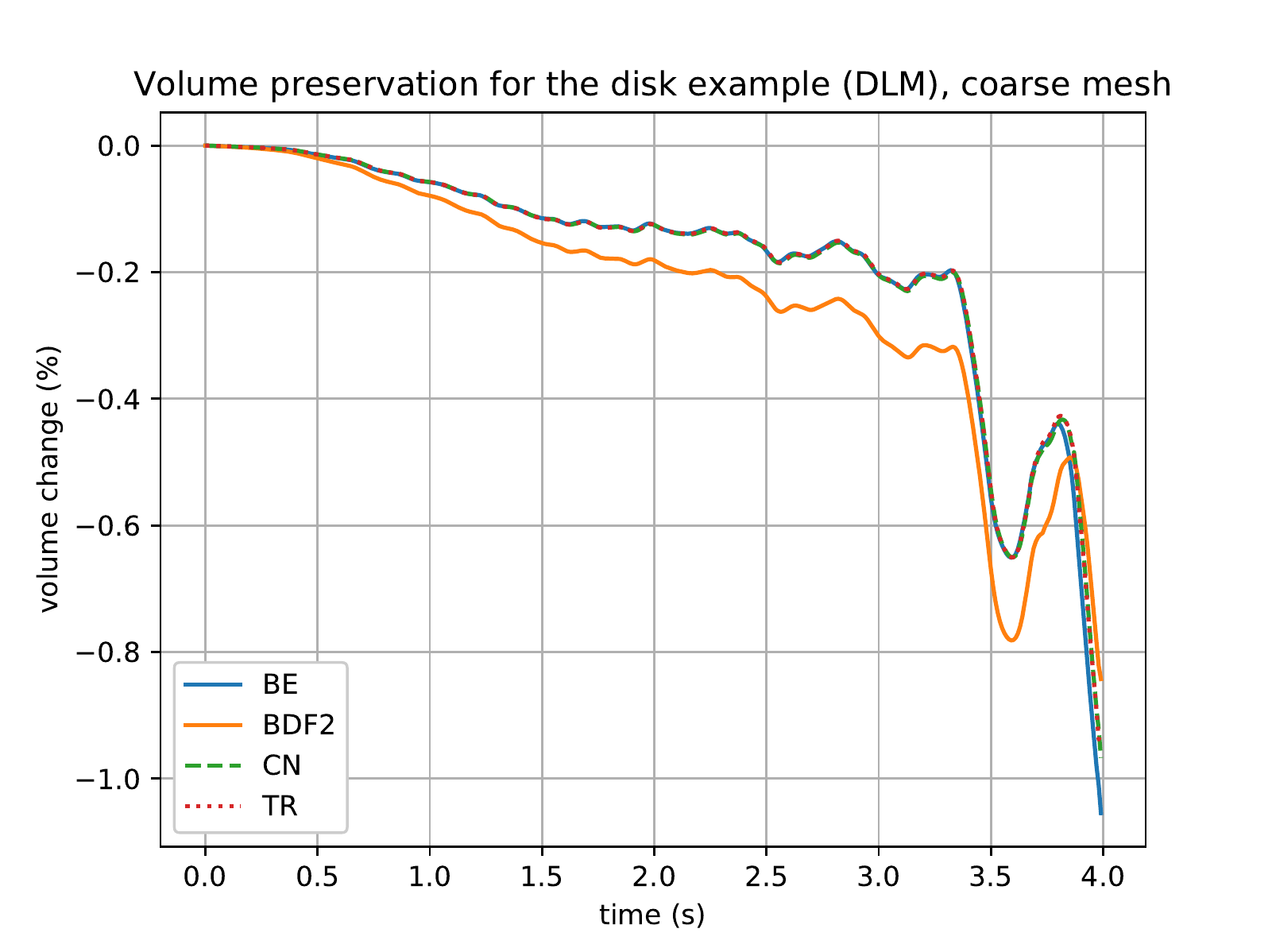}
	\includegraphics[width=0.4\textwidth]{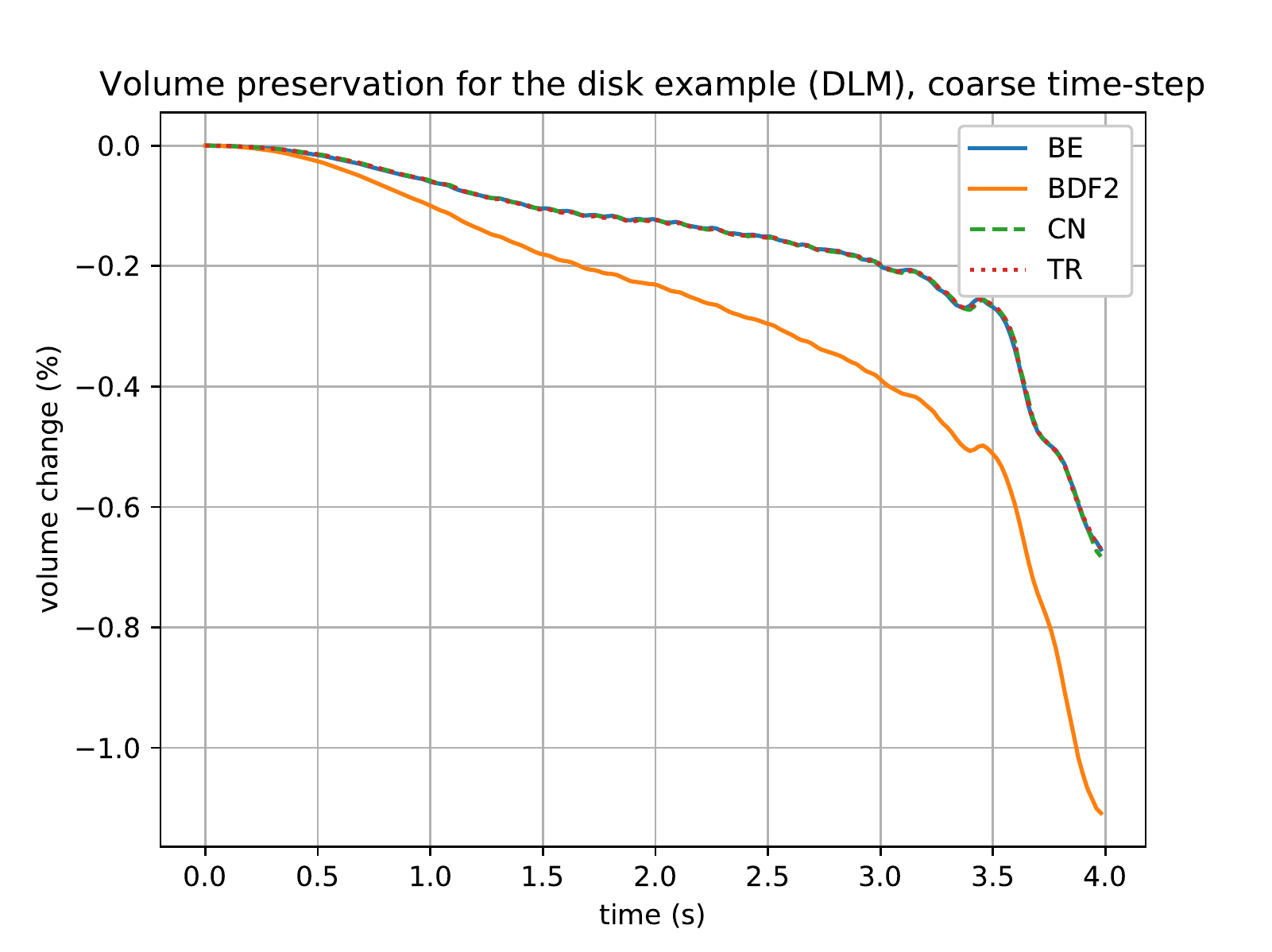}
	\caption{Volume preservation over time for coarser parameters}
	\label{VolumePreservationCoarse-eps-converted-to}
\end{figure} 

With a maximal area change of $-0.7\%$ on the fine mesh with the fine time-step our method performs very well in comparison to other methods. A direct comparison is done with the method developed by Heltai and Costanzo in \cite{HC12}. Their approach is quite similar to ours. They use a finite element approximation in space
based on rectangular meshes and Backward Euler method for the time discretization. Their solver has been run with $33,282$ velocity DOFs, $12,288$ pressure DOFs and $10,370$ DOFs for the displacement field. This resulted in a maximal volume change of $-2.6\%$. A plot over time is shown in Figure \ref{VolumePreservationIFEM}. It is interesting to note that the volume change is monotone for this method, while our method shows some oscillations. This could be due to the different type of meshes and the fact that intersection of a mapped structural element with the elements in the fluid mesh are simpler to detect.   

\begin{figure}
	\centering
	\includegraphics[width=0.7\textwidth]{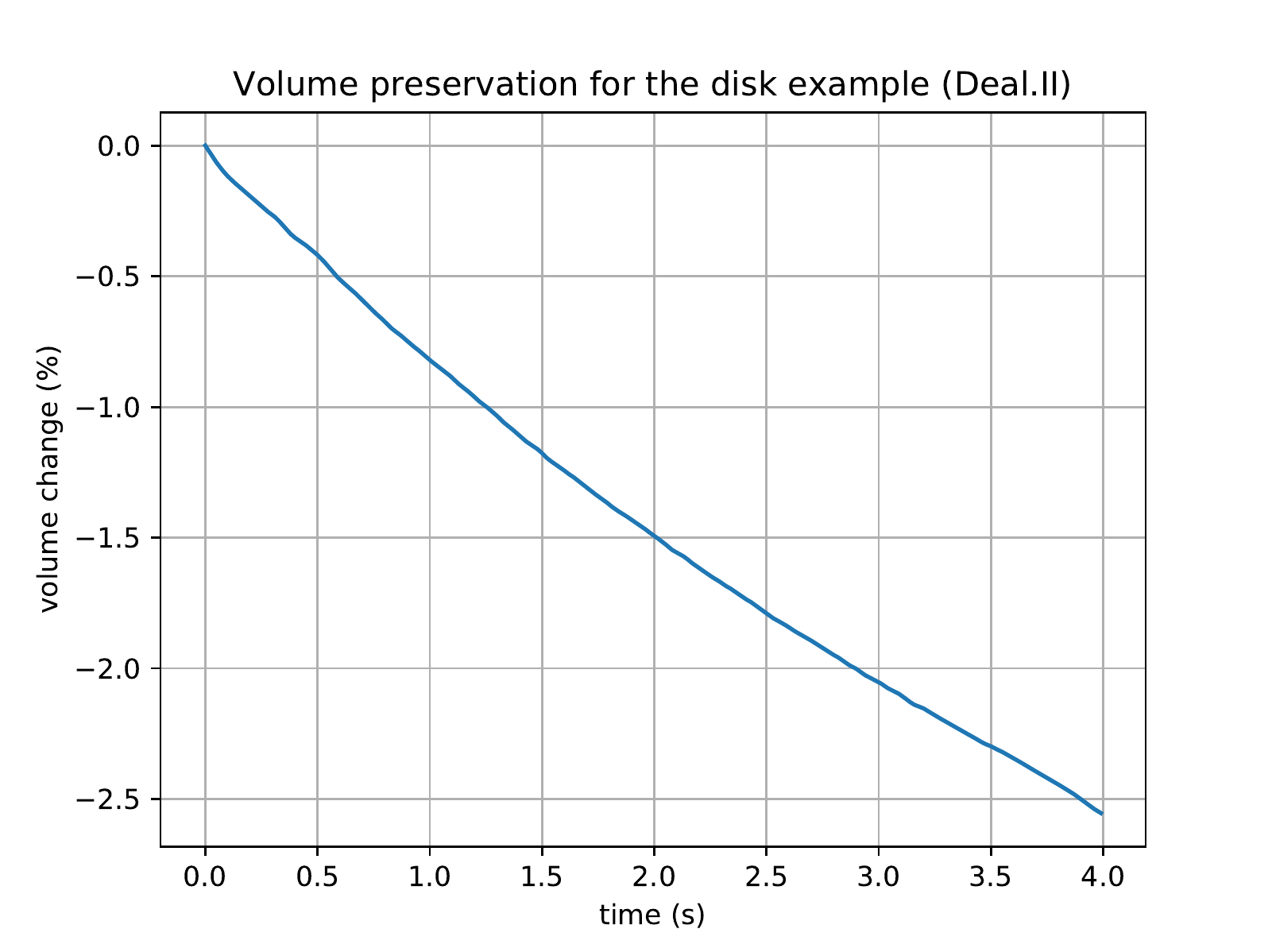}
	\caption{Volume preservation over time for IFEM method}
	\label{VolumePreservationIFEM}
\end{figure} 

\section{Conclusion}

In this paper we discussed a finite element discretization of fluid-structure interaction problems based in the use of a distributed Lagrange multiplier.
We introduced three higher-order time-stepping methods: \BDF, \CNM, and \CNT. We have proved unconditional stability estimates and we have performed a series of numerical tests.
In the numerical experiments the \BDF and \CNT methods showed-higher order convergence in the fully implicit version. Using a semi-implicit approach, the \CNT method showed second order convergence.

\section*{Acknowledgments}
The first and the second authors are members of the INdAM Research group GNCS and their research is partially supported by IMATI/CNR and by PRIN/MIUR

\bibliographystyle{plain}
\bibliography{paper}

\begin{thebibliography}{10}

\bibitem{BCGG12}
D.~Boffi, N.~Cavallini, F.~Gardini, and L.~Gastaldi.
\newblock {Local Mass Conservation of Stokes Finite Elements}.
\newblock {\em J. Sci. Comput.}, 52:383–400, 2012.

\bibitem{BCG11}
D.~Boffi, N.~Cavallini, and L.~Gastaldi.
\newblock {Finite Element approach to Immersed Boundary Method with different
  fluid and solid densities}.
\newblock {\em Math. Models Methods Appl. Sci}, 21(12):2523–2550, 2011.

\bibitem{BCG15}
D.~Boffi, N.~Cavallini, and L.~Gastaldi.
\newblock {The Finite Element Immersed Boundary Method with Distributed
  Lagrange Multiplier}.
\newblock {\em SIAM J. Numer. Anal.}, 53(6):2584--2604, 2015.

\bibitem{BG16}
D.~Boffi and L.~Gastaldi.
\newblock {Discrete models for fluid-structure interactions: The Finite Element
  Immersed Boundary Method}.
\newblock {\em Discrete Contin. Dyn. Syst., Ser. S}, 9:89--107, 2016.

\bibitem{BG17}
D.~Boffi and L.~Gastaldi.
\newblock {A Fictious Domain Approach with Distributed Lagrange Multipliers for
  Fluid-Structure Interactions}.
\newblock {\em Numer. Math.}, 135:711--732, 2017.

\bibitem{BGH07}
D.~Boffi, L.~Gastaldi, and L.~Heltai.
\newblock {Numerical stability of The Finite Element Immersed Boundary Method}.
\newblock {\em Mathematical Models and Methods in Applied Sciences},
  17:1479–1505, 2007.

\bibitem{BGH18}
D.~Boffi, L.~Gastaldi, and L.~Heltai.
\newblock {A distributed Lagrange formulation of the Finite Element Immersed
  Boundary Method for fluids interacting with compressible solids}.
\newblock In Boffi D., Pavarino L., Rozza G., Scacchi S., and Vergara C.,
  editors, {\em Mathematical and Numerical Modeling of the Cardiovascular
  System and Applications}, volume~16 of {\em SEMA SIMAI Springer Series}.
  Springer, 2018.

\bibitem{BGHP08}
D.~Boffi, L.~Gastaldi, L.~Heltai, and C.~S. Peskin.
\newblock {On the hyper-elastic formulation of the immersed boundary method}.
\newblock {\em Comput. Methods Appl. Mech. Eng.}, 197:2210–2231, 2008.

\bibitem{CGSW12}
W.~Chen, M.~Gunzburger, D.~Sun, and X.~Wang.
\newblock {Efficient and long-time accurate second-order methods for
  Stokes-Darcy Systems}.
\newblock {\em SIAM J. Numer. Anal.}, 51(5):2563--2584, 2013.

\bibitem{DB08}
Peter Deuflhard and Folkmar Bornemann.
\newblock {\em Numerische {M}athematik 2}.
\newblock de Gruyter Lehrbuch. [de Gruyter Textbook]. Walter de Gruyter \& Co.,
  Berlin, revised edition, 2008.
\newblock Gew\"{o}hnliche Differentialgleichungen. [Ordinary differential
  equations].

\bibitem{D10}
S.~Dong.
\newblock {BDF-like methods for nonlinear dynamic analysis}.
\newblock {\em J. Comput. Phys.}, 229:3019–3045, 2010.

\bibitem{G12}
B.~E. Griffith.
\newblock {On the Volume Conservation of the Immersed Boundary Method}.
\newblock {\em Commun. Comput. Phys.}, 12(02):401–432, 2012.

\bibitem{GL12}
B.~E. Griffith and X.~Luo.
\newblock {Hybrid finite difference/finite element immersed boundary method}.
\newblock {\em Int. J. Numer. Meth. Biomed. Engng.}, 33(12):e2888, 2012.

\bibitem{HC12}
L.~Heltai and F.~Costanzo.
\newblock {Variational Implementation of Immersed Finite Element Methods}.
\newblock {\em Comput. Methods Appl. Mech. Eng.}, 229-232:110 -- 127, 2012.

\bibitem{HR90}
J.~Heywood and R.~Rannacher.
\newblock {Finite-Element Approximation of the Nonstationary Navier–Stokes
  Problem. Part IV: Error Analysis for Second-Order Time Discretization}.
\newblock {\em SIAM J. Numer. Anal.}, 27(2):353--384, 1990.

\bibitem{IYD17}
O.~R. Isik, G.~Yuksel, and B.~Demir.
\newblock {Analysis of second order and unconditionally stable BDF2-AB2 method
  for the Navier-Stokes equations with nonlinear time relaxation}.
\newblock {\em Numer. Methods Partial Differ. Equations}, 34(6):2060--2078,
  2017.

\bibitem{J16}
V.~John.
\newblock {\em {Finite Element Methods for Incompressible Flow Problems}}.
\newblock Springer, 2016.

\bibitem{OFI10}
Y.~Okamoto, K.~Fujiwara, and Y.~Ishihara.
\newblock {Effectiveness of Higher Order Time Integration in Time-Domain
  Finite-Element Analysis}.
\newblock {\em IEEE Transactions on Magnetics}, 46(8):3321 -- 3324, Aug 2010.

\bibitem{P02}
C.~S. Peskin.
\newblock {The immersed boundary method}.
\newblock {\em Acta Numerica}, 11:479–517, 2002.

\bibitem{RHC15}
S.~Roy, L.~Heltai, and F.~Costanzo.
\newblock {Benchmarking the Immersed Finite Element Method for Fluid-Structure
  Interaction Problems}.
\newblock {\em Comput. Math. Appl.}, 69(10):1167 -- 1188, 2015.

\bibitem{WZ09}
X.~Wang and L.~T. Zhang.
\newblock {Interpolation functions in the immersed boundary and finite element
  methods}.
\newblock {\em Comput. Mech.}, 45(4):321, Dec 2009.

\end{thebibliography}

\end{document}